\documentclass[a4paper]{article}
\usepackage{amsfonts,amssymb}
\usepackage{graphics}
\usepackage{graphicx, epsfig}
\newtheorem{theorem}{Theorem}[section]
\newtheorem{lemma}[theorem]{Lemma}
\newtheorem{cor}[theorem]{Corollary}
\newtheorem{prop}[theorem]{Proposition}
\newtheorem{defn}[theorem]{Definition}
\newtheorem{ques}[theorem]{Question}
\newtheorem{exam}[theorem]{Example}
\newtheorem{rem}[theorem]{Remark}
\newtheorem{crit}[theorem]{Criterion}

\newenvironment{proof}[1][Proof]{\textbf{#1.} }
{\hfill\rule{0.5em}{0.5em}\medskip}
\newenvironment{proof*}[1][Proof]{\textbf{#1.} }{}

\begin{document}

\title{Foliations on non-metrisable manifolds: \\ absorption by a
Cantor black hole}

\author{Mathieu Baillif, Alexandre Gabard and David Gauld\footnote{Supported by the Marsden Fund Council from Government funding, administered by the Royal Society of New Zealand.}
}

\maketitle

\vskip 0.5cm

\smallskip
\newbox\abstract
\setbox\abstract\vtop{\hsize 11cm \noindent

\footnotesize \noindent\textsc{Abstract.} We investigate contrasting behaviours emerging when studying foliations on non-metrisable manifolds. It is shown that Kneser's pathology of a manifold foliated by a single leaf cannot occur with foliations of dimension-one. On the other hand, there are open surfaces admitting no foliations. This is derived from a qualitative study of foliations defined on the long tube $\mathbb S^1\times {\mathbb L}_+$ (product of the circle with the long ray), which is reminiscent of a `black hole', in as much as the leaves of such a foliation are strongly inclined to fall into the hole in a purely vertical way. More generally the same qualitative behaviour occurs for dimension-one foliations on $M \times {\mathbb L}_+$, provided that the manifold $M$ is ``sufficiently small'', a technical condition satisfied by all metrisable manifolds. We also analyse the structure of foliations on the other of the two simplest long pipes of Nyikos, the punctured long plane. We are able to conclude that the long plane $\mathbb L^2$ has only two foliations up to homeomorphism and six up to isotopy.}

\centerline{\hbox{\copy\abstract}}

\bigskip


\bigskip
\normalsize

%\section{Introduction}
%\section{Introduction}
%\section{Introduction}
%\section{Introduction}
\section{Introduction}\label{sec1}

All of our manifolds are assumed to be non-empty, connected, Hausdorff
spaces in which each point has a neighbourhood homeomorphic to
Euclidean space $\mathbb R^n$ of some fixed dimension $n$. We note in passing that there are many conditions equivalent to metrisability of a manifold, including paracompactness, Lindel\"ofness and second countability. 

We recall that there are four manifolds of dimension $1$: the
circle $\mathbb S^1$, the real line $\mathbb R$, the \emph{long
ray} $\mathbb L_+$ and the \emph{long line} $\mathbb L$ (apparently
this classification was first worked out by Hellmuth Kneser
\cite{HellmuthKneser58}). The spaces $\mathbb L_+$ and
$\mathbb L$ are, respectively, the interior and double of the
\emph{closed long ray}, denoted by $\mathbb L_{\ge0}$ and constructed
through `continuous interpolation' of the first uncountable
ordinal $\omega_1$, that is to say as $\omega_1 \times [0,1)$
topologised by the lexicographical order. The idea originates
with Cantor \cite[p.\,552]{Cantor}, reappears in an
unpublished `Nachlass' of Hausdorff \cite[pp.\,317--318]{Hausdorff15}, Vietoris
\cite[pp.\,183--184]{Vietoris}, Alexandroff \cite[footnote
p.\,295]{Alexandroff} and the Knesers \cite{KneserKneser60},
Spivak \cite{Spivak} or Nyikos \cite{Nyikos}.

While compact, indeed metrisable, 2-manifolds have been classified, there is little hope of classifying the non-metrisable 2-manifolds. However, there is the Bagpipe Theorem of Nyikos \cite[Theorem 5.14]{Nyikos} which states that every $\omega$-bounded 2-manifold is obtained from a closed surface by removing finitely many disjoint discs and replacing them by long pipes. Following Nyikos we define \emph{$\omega$-bounded} to mean that every countable subset has compact closure and a \emph{long pipe} to be the union of a chain $\langle U_\alpha\ :\ \alpha<\omega_1\rangle$ of open subspaces each homeomorphic to $\mathbb S^1\times\mathbb R$ such that $\overline{U_\alpha}\subset U_\beta$ and that the frontier of $U_\alpha$ in $U_\beta$ is homeomorphic to $\mathbb S^1$ when $\alpha<\beta$.

The literature contains a range of definitions of a foliation, especially on a metrisable manifold. When it comes to non-metrisable manifolds one needs to be more careful, particularly, in view of Kneser's example of a non-trivial foliation with a single leaf, one must avoid definitions involving partitions. We will adopt the following definition, which goes back to Reeb's thesis \cite{Reeb} and is quite close to that of Milnor \cite{Milnor70}, where some of the issues arising in the non-metrisable case are discussed.

\begin{defn}\label{define foliation}
A \emph{foliation} $\mathcal F$ on a manifold $M^n$ is a maximal
atlas $\{(U_i,\varphi_i):i\in I\}$ on $M$ such that for
each $i,j\in I$ the coordinate transformations
$\varphi_j\varphi_i^{-1}:\varphi_i(U_i\cap
U_j)\to\varphi_j(U_i\cap U_j)$ are of the form
$$\varphi_j\varphi_i^{-1}(x,y)=\Big(g_{i,j}(x,y),h_{i,j}(y)\Big)$$
for all $(x,y)\in\mathbb R^{p}\times\mathbb R^{n-p}$, where
$g_{i,j}:\varphi_i(U_i\cap U_j)\to\mathbb R^{p}$ and $h_{i,j}$ is an embedding from a relevant open subset of $\mathbb R^{n-p}$ to $\mathbb R^{n-p}$. Call such a chart a \emph{foliated chart}.
Components of sets of the form $\varphi_i^{-1}(\mathbb R^{p}\times\{y\})$
are called \emph{plaques}.
The latter constitute the basis for a new topology on $M^n$
(known as the {\em leaf topology}) whose path components are
injectively immersed $p$-manifolds called the
\emph{leaves} of $\mathcal F$. The number $p$ is the
\emph{dimension} of $\mathcal F$ while $n-p$ is the \emph{codimension}.
\end{defn}

Our primary goal is to study foliations on non-metrisable manifolds. As we shall see, the shift to non-metrisable foliated manifolds is ``two-fold'' producing both regularities and anomalies. By the former we mean that sometimes foliation theory on certain (non-metrisable) manifolds turns out to collapse to a very rigid art form, obeying some crystallographic patterns, hardly expectable from the plasticity observed in the metrisable realm. On the other hand, some curious phenomena can happen when metrisability is dropped, including a codimension-one foliation on a non-metrisable $3$-manifold possessing only a single leaf. This was first pointed out by Martin Kneser \cite{MartinKneser}\footnote{Actually Martin Kneser presented his example as a dimension raising continuous bijective map from a surface to a $3$-manifold, a bit in the spirit of Peano's curve, and his father Hellmuth interpreted the example in terms of foliations (see
\cite{H.Kneser62}).}, mentioned in Haefliger \cite{Haefliger62} and popularised by Milnor~\cite{Milnor70}. This contrasts with the metrisable case, where the set of leaves is at least uncountable; indeed has exactly the power of the continuum. (This follows from the fact that each leaf endowed with its (fine) leaf topology is still second countable, see \cite{Haefliger55}.) One may then wonder if Kneser's pathology already occurs on surfaces. A negative answer is included in:

\begin{theorem}\label{Many Leaves}
A dimension-one foliation on a (not necessarily metrisable) manifold of dimension at least $2$ has exactly $\mathfrak{c}={\rm card}({\mathbb R})$ many leaves.
\end{theorem}

We shall present a visual approach to Kneser's example in Section \ref{Visualising Kneser's example} and prove Theorem~\ref{Many Leaves} in Section \ref{sec3}.

Because of the role played by long pipes in Nyikos's Bagpipe Theorem we spend some time looking at foliations on long pipes. We firstly require some basic results which we gather in Section \ref{Basic results} for later reference. 

A unifying feature observed in the long pipes and their generalisations which we consider is a product structure in which one factor is a metrisable manifold $M$ and the other is $\mathbb L_+$. This product structure manifests itself in any foliation yielding a kind of asymptotic rigidity; we think of this as a kind of black hole behaviour. We make this more precise in Section \ref{Black holes}, where we present an analysis of dimension-one foliations on manifolds of the form $M\times\mathbb L_+$.

One of the simplest long pipes is $\mathbb S^1\times\mathbb L_+$. We prove the following and related results in Section \ref{Black hole consequences}.

\begin{theorem}\label{black hole} A dimension-one foliation
$\mathcal{F}$ on $\mathbb{S}^1 \times \mathbb{L}_+$
is confined to the following (mutually exclusive) alternatives:
\begin{itemize}
\item[{\rm (i)}] either the set $C=\{\alpha\in\mathbb L_+:\mathbb
S^1\times\{\alpha\} \mbox{ is a leaf of } \mathcal F\}$ is a closed unbounded
subset of $\mathbb L_+$, or

\item[{\rm (ii)}] the foliation is ultimately vertical, i.e. there
is an ordinal $\alpha\in\omega_1$ such that the restricted
foliation on $\mathbb S^1\times (\alpha,\omega_1)$ is the
trivial product foliation by long rays.
\end{itemize}
\end{theorem}

Picturesquely, the leaves (thought of as light rays) are inclined to fall into the black hole in a purely vertical way due to the strength of gravitation (Case (ii)), but sometimes they manage to resist the huge attraction by winding fast around it (Case (i)).

Recall that in the classical metrisable setting, the existence of foliations is well understood in codimension-one (at least if smoothness is assumed). In the open case
existence is systematic, while in the closed case there is a single obstruction, the non-vanishing of the Euler characteristic (Thurston \cite{Thurston76}).

The existence of a codimension-one foliation on any smooth metrisable open manifold $M$ reduces to the existence of a differentiable function $f\colon M \to {\mathbb R}$ without critical points, therefore a submersion whose level hypersurfaces generate the desired foliation. Such a function $f$ is obtainable by eliminating the (isolated) critical
points of a Morse function; either by using Whitehead's spines as in Hirsch \cite[Theorem 4.8]{Hirsch61}, or by pushing them to infinity along arcs, as in Godbillon \cite[p.\,9]{Godbillon91}.

In some sense the latter method is better, since avoiding triangulations, it allows a clearcut extension to topological manifolds:

\begin{theorem}\label{open metrisable}
Any open metrisable topological manifold carries a
codimension-one foliation, which is definable by a single
global real valued topological submersion.
\end{theorem}

\begin{proof}
Recall first that metrisable topological manifolds of dimension $n \neq
4$ admit handlebody decompositions: for $n\ge 6$ see
Kirby-Siebenmann \cite[p.\,104]{KirbySiebenmann}; while the case $n=5$ is settled by Quinn in \cite{Quinn82} (see also
\cite[p.\,135-6]{FreedmanQuinn}). Such a decomposition is the
same as a topological Morse function, which can then be
improved to a submersion by the trick above of pushing
`troubles' to $\infty$. The remaining case $n=4$ cannot be so handled due to the disruption of
handlebody theory (see Siebenmann \cite{Siebenmann70}, updated
by Freedman \cite{Freedman82}, to locate $4$ as a disrupting
dimension). However one can get around the disruption by quoting another result of Quinn, namely the smoothability of
open metrisable 4-manifolds (cf. \cite[p.\,116]{FreedmanQuinn}).
\end{proof}

The hope encouraged by Theorem \ref{open metrisable} that non-metrisable manifolds might all admit foliations is not borne out. \iffalse Extrapolating from compact through non-compact to non-metrisable manifolds, we might guess that more room provides more freedom to construct foliations but this turns out to be misleading in the non-metrisable case.\fi We will introduce a class of surfaces $\Lambda_{g,n}$ obtained from the compact surface of genus $g$ as in the Bagpipe Theorem where there are $n$ long pipes, all homeomorphic to $\mathbb S^1\times\mathbb L_+$, (see Figure~9 below). As a corollary to Theorem \ref{black hole} we will show the following in Section \ref{Black hole consequences}.

\begin{cor}\label{non-foliable surfaces} 
None of the surfaces ${\Lambda}_{g,n}$ admit
foliations except for ${\Lambda}_{0,2}$ the sphere with two black holes (homeomorphic to the `doubly' long cylinder $\mathbb S^1 \times {\mathbb L}$) and ${\Lambda}_{1,0}$ the torus without any black hole.
\end{cor}

A result in this direction was already mentioned in a paper of Nyikos \cite[p.\,275]{Nyikos79}, but unfortunately neither a detailed proof nor a description of the surfaces was given.

The analogue of Theorem \ref{black hole} for $\mathbb S^2
\times {\mathbb L}_{+}$ also holds but with a difference. As
the base manifold $\mathbb S^2$ has no dimension-one
foliation, this prompts the more cannibalistic behaviour that
each one-dimensional foliation is ultimately vertical. This is
the object of Corollary \ref{cannibal}.

Of course a four-dimensional black hole $\mathbb S^3\times
{\mathbb L}_+$ is not cannibalistic, because $\mathbb S^3$
admits one-dimensional foliations, e.g. the one given by the
Hopf fibration $\mathbb S^3 \to {\mathbb C} P^1$.
In fact in this case we obtain an exact analogue of Theorem
\ref{black hole}, except that in (i) the condition ``{\it $\mathbb S^1\times\{\alpha\} \mbox{ is a leaf of } \mathcal F$}'' is replaced by ``{\it  $\mathbb S^3\times
\{\alpha \}$ is foliated by $\mathcal F$}''. Here of course this
does not imply the presence of a circular leaf: remember
Schweitzer's negative solution to the ``Seifert conjecture''
\cite{Schweitzer74}.

Our understanding of higher-dimensional foliations on
long-objects like $M\times {\mathbb L}_+$ is much more
fragmentary. Again one might guess that longness imposes some kind of
rigidity in the large (say ultimately one can only observe a
product foliation or eventually a transfinite gluing of a
foliation on $M\times[0,1]$). For instance it would be
interesting to understand better codimension-one foliations on
$\mathbb S^2\times {\mathbb L}_+$. Do they always
exhibit a spherical leaf? Do they never exhibit compact leaves
of genus $\ge 2$?

\medskip

It could be also interesting to ask whether for each integer $n\ge 2$ one can find an open $n$-manifold supporting no codimension-one foliations (or even no foliation at all). A natural candidate is the $n$-dimensional ``long glass'' ${\Lambda}^n={\mathbb B}^n \cup_{\partial} (\mathbb S^{n-1} \times {\mathbb L}_{\ge 0})$, an $n$-dimensional version of the surface $\Lambda_{0,1}$ considered in Corollary~\ref{non-foliable surfaces}. The manifold $\Lambda^n$ is hard to foliate, because it is essentially impossible to foliate the $n$-ball, $\mathbb B^n$ (compare Proposition \ref{prop7.6} for a precise statement).

This, albeit rather slim evidence, leads us to put forward the following:

\smallskip
\noindent {\bf Speculation.} {\it The $n$-dimensional long glass ${\Lambda}^n$
supports no $C^0$-foliations (except the two trivial ones).
}
\smallskip

Our methods validate this for dimension-one foliations (see Corollary \ref{longglass}): giving examples in each dimension $n\ge 2$ of an open $n$-manifold without
dimension-one foliations.

Another of the simplest long pipes is the punctured long plane, $\mathbb L^2-\{{\rm pt}\}$. Our effort is concentrated on the behaviour of dimension-one foliations `towards infinity'. Interestingly, $\mathbb L^2-\{{\rm pt}\}$, or more generally $\mathbb L^2-K$ for some compactum $K$, splits naturally into pieces which, while not themselves being products of the form $M\times\mathbb L_+$, have enough of the structure of this product for us to be able to apply the results of Section \ref{Black holes}. We find that there are six different cases, described in Section~\ref{Foliating L2-pt}. Filling the holes in these six asymptotic structures enables one to deduce a complete classification of foliations on the long plane ${\mathbb L}^2$, which up to homeomorphism contains only \emph{two} representatives : the trivial foliation and the ``broken'' foliation, in which leaves switch from the vertical to the horizontal mode when they cross the diagonal. So from the foliated viewpoint the ``Cantorian'' plane ${\mathbb L}^2$ appears as an extremely ``rigid'' object, when compared to its Euclidean analogue ${\mathbb R}^2$, which in contrast allows a (huge) menagerie of foliations by the celebrated works of Kaplan~\cite{Kaplan40} and Haefliger-Reeb~\cite{HaefligerReeb57}.

So far we have studied foliations \emph{per se}, without worrying much about applications. However the general philosophy that geometric structures (in particular
foliations) aid a better understanding of the underlying manifold topology is very effective in our setting too. Concretely, foliations provide a sensitive medium to distinguish among similar looking manifolds, whose nuances remain undetected through
the eyes of classical algebro-topological invariants. Such a situation occurs when it comes to distinguish the rectangular from the rhombic quadrants (see Corollary \ref{rectangular v rhombic}).

%Revisiting Kneser's example
%Revisiting Kneser's example
%Revisiting Kneser's example
\section{Visualising Kneser's example}\label{Visualising Kneser's example}

The example presented by Kneser in \cite{MartinKneser} is
related to the non-metrisable surface introduced by Pr\"ufer
in 1923 and first described in print by Rad\'o \cite{Rado}.
We first recall Pr\"ufer's construction in order to get an
intuitive picture of such a surface. This being done it is
then easy to visualise explicitly how the pathology of a
single leaf can come about.

We use the complex numbers ${\mathbb C}$ as model for the Euclidean plane. The idea
is to consider the set $P$ formed by the (open) upper
half-plane ${\mathbb H} = \{ z : {\rm Im }(z)>0 \}$ together
with the set of all rays emanating from points of ${\mathbb R}$
and pointing into the upper half-plane. We topologise $P$
with the usual topology for $\mathbb H$, and by
declaring as neighbourhoods of a point $r$ which is a ray (say
emanating from $x\in {\mathbb R}$) an (open) sector of rays
deviating by at most $\varepsilon$ radian from $r$, together
with the points of $\mathbb H$ lying between the two rays,
while remaining at (Euclidean) distance smaller than
$\varepsilon$ from $x$ (see Figure \ref{PruferRay}).

%\scalebox{0.5}{\includegraphics{nicebristlysphere.pdf}}
%\caption{A bristly sphere.}\label{Bristly Sphere}

\begin{figure}[h]
\hspace{-1cm}
\begin{minipage}[b]{0.5\linewidth} % A minipage that covers half the page
\centering
\scalebox{1.0}{\includegraphics{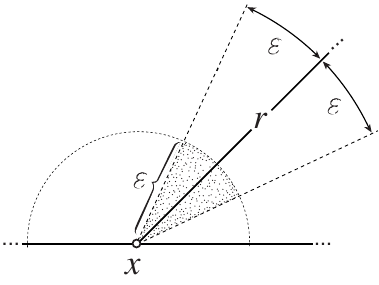}}
\caption{\label{PruferRay} A neighbourhood of a ray}
\end{minipage}
\hspace{-1cm} % To get a little bit of space between the figures
\begin{minipage}[b]{0.5\linewidth}
\centering
\scalebox{1.0}{\includegraphics{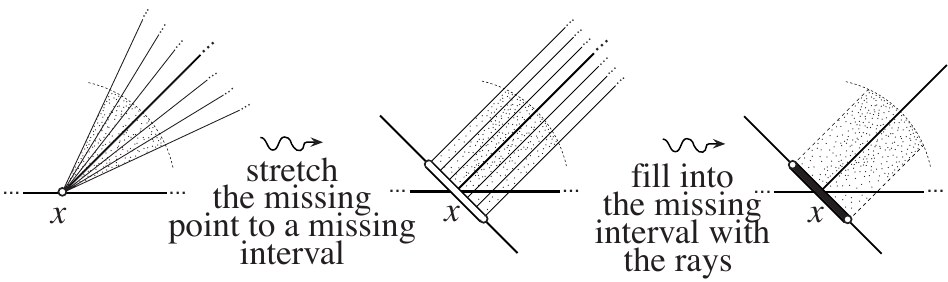}}
 \caption{\label{PruferMan} Proving that $P$ is a
surface-with-boundary}
\end{minipage}
\end{figure}

%Old version of the figure that didn't work with arXiv cut to........
\iffalse
\begin{figure}[h]
\hspace{-1cm}
\begin{minipage}[b]{0.5\linewidth} % A minipage that covers half the page
\centering
\epsfig{figure=PruRay.pdf,width=40mm}
\caption{\label{PruferRay} A neighbourhood of a ray}
\end{minipage}
\hspace{-1cm} % To get a little bit of space between the figures
\begin{minipage}[b]{0.5\linewidth}
\centering
\epsfig{figure=PruMan.pdf,width=100mm}
 \caption{\label{PruferMan} Proving that $P$ is a
surface-with-boundary}
\end{minipage}
\end{figure}
\fi
%...................here.

The space $P$ is a surface-with-boundary, as heuristically
explained by Figure \ref{PruferMan}. (A more careful
discussion of this point may be found e.g. in
\cite{Gabard08}.) Observe that $P$ has a continuum ${\frak c}$
of boundary components each homeomorphic to the real line
${\mathbb R}$. So one may think of $P$ as being just the
(closed) upper-half-plane,
with each boundary point blown up to a copy of the real line.

We are now ready to describe a foliated structure on a
3-manifold having just a single leaf. We first consider the
product $W^3=P \times {\mathbb R} \ni (z,t)$  which is a
3-manifold with boundary components $C_x$ indexed by the
reals, each homeomorphic to ${\mathbb R}^2$. From it we construct
a 3-manifold $M^3$ by identifying for each $x>0$ the boundary
components $C_{-x}$ and $C_{x}$ through a translation of
amplitude $x$ in the $t$-coordinate, and adding an (open)
collar to the central component $C_0$. We foliate $M^3$ by the
`vertical' ($t=$constant) plane field (see Figure \ref{OneLeaf}).

\begin{figure}[h]
\centering
\scalebox{1.1}{\includegraphics{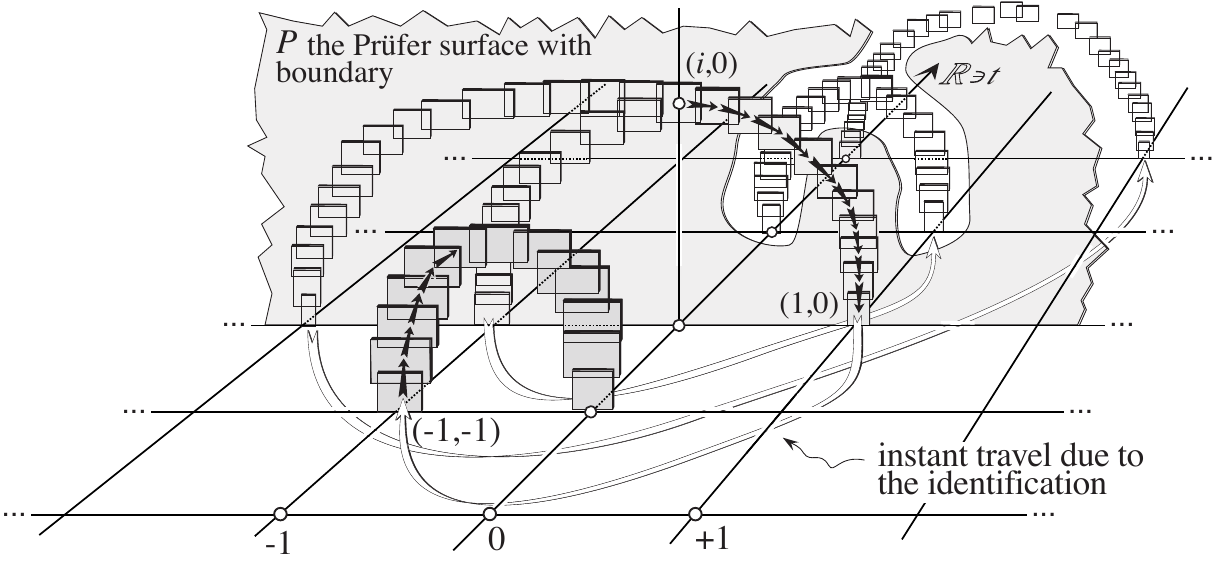}}
    \caption{\label{OneLeaf} A non-metrisable 3-manifold foliated by a single leaf}
\end{figure}

%Old version of the figure that didn't work with arXiv cut to........
\iffalse
\begin{figure}[h]
\centering
    \epsfig{figure=OneLeaf.pdf,width=130mm}
    \caption{\label{OneLeaf} A non-metrisable 3-manifold foliated by a single leaf}
\end{figure}
\fi
%..................here.

Starting say from the point $(i,0)$ ($i=\sqrt{-1}$), and travelling as indicated by the arrows on Figure~3, one crosses the (ex-)boundary at $(1,0)$ to reappear in $(-1,-1)$ (due to the
identification). So by varying the position at which one decides to cross the boundary, one can vary as one pleases the $t$-coordinate of the reappearance. Therefore one may reach any other point of the manifold $M^3$ by moving only in the restricted way prescribed by the foliation. This shows that the envisaged foliation has just a single leaf.

\begin{rem} {\rm Kneser's example shows that the concept of a foliation cannot (in general) be reduced to the single data of a partition of the manifold satisfying certain conditions however stringent they might be. So at a foundational level it is certainly quite important forcing us to work with the modern definition of a foliation as the geometric structure associated with a suitable pseudo-group. In the metrisable case the partition point of view is equivalent to Definition \ref{define foliation} in the sense that the function taking a foliation to the partition into leaves is injective. We have found no formally published reference of this fact but \cite[Lemma 7]{vdBan} does give a complete proof.}\end{rem}

%Foliations of dimension-one have many leaves
%Foliations of dimension-one have many leaves
%Foliations of dimension-one have many leaves
%Foliations of dimension-one have many leaves
\section{Foliations of dimension-one have many leaves} \label{sec3}

In this section we show that the anomaly (of Section \ref{Visualising Kneser's example}) of a single leaf filling up the whole manifold `ergodically' cannot occur if the ambient dimension is only $2$. The reason behind this well-behaviour is not specifically two-dimensional, but rather to be found in the one-dimensionality of the leaves, particularly the fact that 1-manifolds are completely classified. We prove the more general Theorem
\ref{Many Leaves}.
\medskip

{\bf Proof of Theorem \ref{Many Leaves}.} Let ${\mathcal L}({\mathcal F})$ denote the set of leaves of the dimension-one foliation ${\mathcal F}$ on the $n$-manifold $M$ and set $\lambda={\rm card} {\mathcal L}({\mathcal F})$: to show $\lambda=\mathfrak c$. 

The obvious surjection $M \to {\mathcal L}({\mathcal F})$ shows that ${\frak c}\ge\lambda$ because the cardinality of non-trivial connected, Hausdorff manifolds is always that of the continuum (see \cite[Problem 8, p.\,A-15--A-16]{Spivak} or \cite[Theorem 2.9]{Nyikos}). 

Let $\varphi\colon U \to {\mathbb R}^n$ be a foliated chart for ${\mathcal F}$ with $\varphi(U)={\mathbb R}^n$, so that $P_y=\varphi^{-1}({\mathbb R }\times \{ y \})$ with $y\in {\mathbb R}^{n-1}$ are the corresponding plaques. One has an `integration' map ${\mathcal P}:=\{P_y \}_{y \in {\mathbb R}^{n-1}} \to {\mathcal L}({\mathcal F})$ taking each plaque to its leaf extension.

It suffices to show that each leaf of $\mathcal F$ contains only countably many plaques of $\mathcal P$. Indeed, in that case one can find an injection ${\mathcal P}\hookrightarrow {\mathcal L}({\mathcal F}) \times {\mathbb N}$. Since $n\ge 2$, this gives  ${\mathfrak c}={\rm card} (\mathcal P) \le \lambda \cdot \omega$, and hence ${\mathfrak c} \le \lambda$.

Suppose for a contradiction that there is a leaf $L$ containing uncountably many plaques of $\mathcal P$. In view of the classification of $1$-manifolds the leaf $L$ is either ${\mathbb L}$ or ${\mathbb L}_{+}$, because the two separable manifolds $\mathbb S^1$ and ${\mathbb  R}$ cannot contain uncountably many pairwise disjoint open sets.

Let us first assume $L\approx{\mathbb L}$. The uncountable subset $\{ y \in
{\mathbb R}^{n-1} : P_y \subset L \}$ of ${\mathbb R}^{n-1}$ has a \emph{condensation point}, i.e. a point of the set which is the limit point of a non-stationary sequence of points of the set. Hence, one finds inside $L$ a point $x\in U$ which is the limit of a converging sequence $\langle x_n\rangle$ of points of $U$, none of which belongs to the plaque through $x$ (compare Figure \ref{ManyLeaves}).

\begin{figure}[h]
\centering
\scalebox{1.1}{\includegraphics{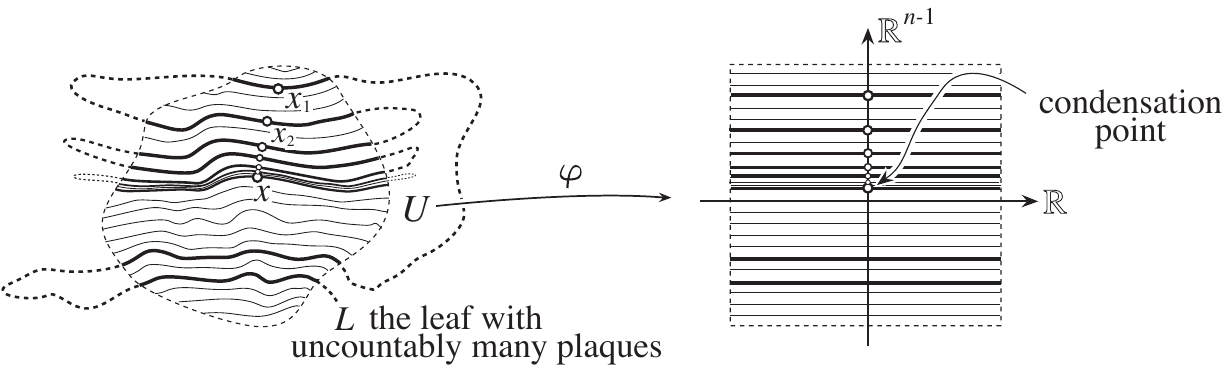}}
    \caption{\label{ManyLeaves} Finding many leaves}
\end{figure}

%Old version of the figure that didn't work with arXiv cut to........
\iffalse
\begin{figure}[h]
\centering
    \epsfig{figure=ManyLea.pdf,width=125mm}
    \caption{\label{ManyLeaves} Finding many leaves}
\end{figure}
\fi
%..................here.

Since the long line ${\mathbb L}$ is sequentially compact, taking a subsequence if necessary, we may assume that $\langle x_n\rangle$ converges also in the leaf topology on $L$ (say to $\widetilde{x}$). Note that $\widetilde{x}\neq x$, because the plaque through $x$ does not contain any member of the sequence $\langle x_n\rangle$. Since the leaf topology on $M$ is a refinement of its usual topology, it follows that $\langle x_n\rangle$ converges to $\widetilde{x}$ as well in the usual topology on $M$. This contradicts the uniqueness of the limit in Hausdorff spaces.

Finally, if $L \approx {\mathbb L}_{+}$, one finds a point $p\in L$ not in any of the plaques of $U$ lying in $L$. The short side of $L- \{ p \}$ can absorb at most countably many plaques, and arguing exactly as before one can contradict the assumption that there are uncountably many plaques in the long side of $L- \{ p\}$ (think of this as closed by adding $p$, to make it sequentially compact). \hfill\rule{0.5em}{0.5em}
\medskip

The same argument shows that if all leaves of a codimension $>0$ foliation are sequentially compact then there are $\mathfrak c$ many leaves. It also shows that leaves modelled on $\mathbb L$ are embedded.

%{Basic Results}
%{Basic Results}
%{Basic Results}
%{Basic Results}
\section{Basic results} \label{Basic results}

In this section we gather some useful facts. The first is a standard property of the order topology on $\omega_1$.

\begin{crit}\label{closed in omega_1}
A subset $C\subset\omega_1$ is closed if and only if every increasing sequence in $C$ converges in $C$.
\end{crit}

\begin{defn}
Call a topological space $X$ \emph{squat} provided that every
continuous map $f:\mathbb L_+\to X$ is \emph{eventually
constant}, i.e. there are $\beta\in\mathbb L_+$ and $x\in X$
so that $f(\alpha)=x$ for each $\alpha\ge\beta$.
\end{defn}

Our first lemma generalises the well-known fact that $\mathbb R$ is squat (consult \cite[Satz 3]{KneserKneser60} or \cite[Lemma 3.4 (iii)]{Nyikos}): indeed, the lemma implies that all metrisable manifolds are squat.

\begin{lemma}\label{Mathieu}
If a space $X$ is first countable, Lindel\"of and Hausdorff then it is squat.
\end{lemma}

\proof We will prove it using the graph of $f:{\mathbb L}_+\to X$
and thus work in ${\mathbb L}_+\times X$. The graph $\Gamma_f$ of
$f$ is closed (because $X$ is Hausdorff) and ${\mathbb L}_+$-{\it unbounded} (i.e.
its projection on the ${\mathbb L}_+$-factor is unbounded).

We shall use the:

\smallskip
\noindent {\bf Sublemma.} {\it Let $X$ be a space as above and
$C\subset {\mathbb L}_+\times X$ be closed and ${\mathbb
L}_+$-unbounded. Then there is $x\in X$ so that $C\cap({\mathbb
L}_+\times \{x\})$ is ${\mathbb L}_+$-unbounded.}
\smallskip

We apply this to $C=\Gamma_f$. Since the ${\mathbb
L}_+$-projection of $\Gamma_f\cap({\mathbb L}_+\times  \{x\})$ is
nothing but $f^{-1}(x)$, we conclude that the latter is a closed unbounded set.

Let $(V_n)_{n\in {\mathbb N}}$ be a countable fundamental system of
open neighbourhoods of $x$ so that $\cap_{n} V_n=\{ x\}$, and
consider the closed subsets $f^{-1}(X-V_n)\subset {\mathbb L}_+$.
The latter are disjoint from $f^{-1}(x)$ and therefore
bounded (recall that two closed unbounded subsets of ${\mathbb L}_+$ always
intersect). Hence $f^{-1}(X-\{x\})=\cup_{n\in{\mathbb N}}f^{-1}(X-V_n)$ is bounded
as well; beyond this bound $f$ can take only the value $x$. This completes the proof of the lemma.

\medskip
{\bf Proof of the Sublemma.} If not, then for all $x\in X$,
$C\cap ({\mathbb L}_+\times \{ x \})$ is ${\mathbb L}_+$-bounded. So
there is a $\beta_x\in {\mathbb L}_+$ such that
$[\beta_x,\omega_1)\times  \{x\}$ does not intersect $C$. Fix some
$x\in X$. Then there is an $n$ so that the ``thickening''
$[\beta_x,\omega_1)\times V_n$ still does not meet $C$. If not
construct a sequence $\langle x_n\rangle$ by choosing points $x_n \in
\bigl([\beta_x,\omega_1)\times V_n \bigr) \cap C $ which due
to the sequential compactness of $[\beta_x,\omega_1)$ can be
assumed to be convergent (extracting a subsequence if
necessary). The limiting point $x_{\omega}$ would belong to
$\bigl([\beta_x,\omega_1)\times  \{x\}\bigr) \cap C$. A
contradiction.

Now let $x$ vary, and denote the $V_n$ above more accurately
by $V_{n(x)}^{x}$. The $(V_{n(x)}^{x})_{x\in X}$ form an open
cover of $X$. By Lindel\"ofness we may extract a countable
subcover corresponding to some countable subset $N$ of $X$.
Then $\beta=\sup_{x\in N} \beta_x$ is an ${\mathbb L}_+$-bound
for $C$. This contradiction proves the sublemma. \endproof

{\it Note.} None of the assumptions in Lemma \ref{Mathieu} can be removed. First countability cannot be relaxed (choose as $X$ the one-point compactification of ${\mathbb L}_+$), Lindel\"of is of course essential (take $X={\mathbb L}_+$) and finally Hausdorffness cannot be completely omitted (take for $X$ a space with coarse topology of cardinality at least two).
\smallskip

Even for manifolds we cannot expect a simple converse of Lemma \ref{Mathieu} as the following example shows.

\begin{exam}Both classical versions of the Pr\"ufer surface are squat.
\end{exam}

Both versions of the Pr\"ufer surface we are referring to are obtained from the surface-with-boundary $P$ introduced in Section \ref{Visualising Kneser's example} either by adding open collars or by doubling it (see Figure \ref{PruferSquat}); denote either here by $\widehat P$.

To verify squatness of both versions of the Pr\"ufer surface, given a continuous map $f:\mathbb L_+\to\widehat P$, follow it by projection onto the $y$-coordinate. As this gives a map into $\mathbb R$, it follows that the $y$-coordinate of $f$ is eventually constant, which then implies that eventually $f$ maps into a copy of $\mathbb R$ and hence $f$ is eventually constant.

\begin{figure}[h]
\centering
\scalebox{1.1}{\includegraphics{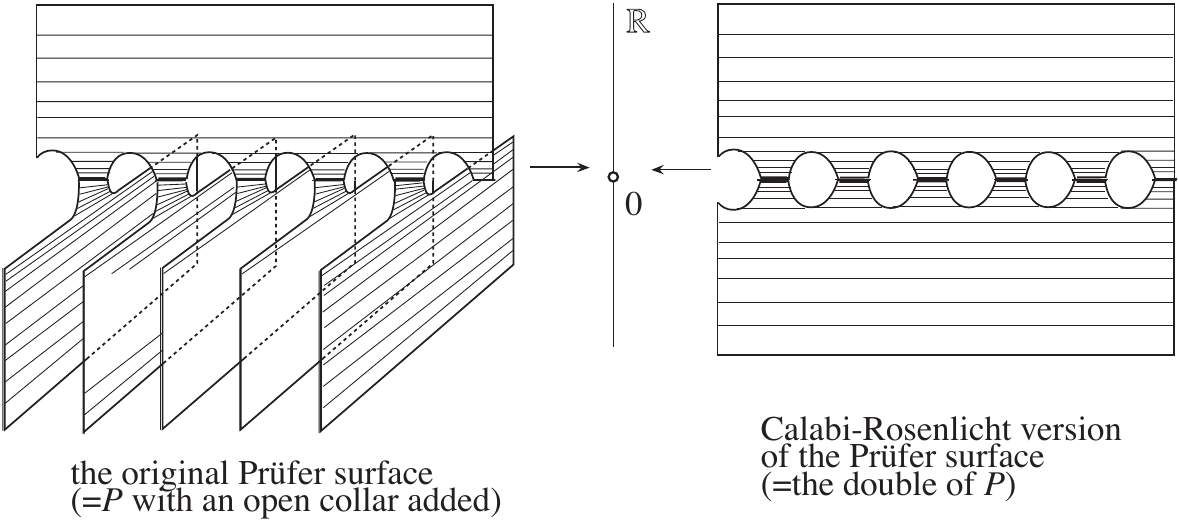}}
    \caption{\label{PruferSquat} Pr\"ufer surfaces are squat}
\end{figure}

%Old version of the figure that didn't work with arXiv cut to........
\iffalse
\begin{figure}[h]
\centering
    \epsfig{figure=PrSquat.pdf,width=125mm}
    \caption{\label{PruferSquat} Pr\"ufer surfaces are squat}
\end{figure}
\fi
%..................here.

\medskip
Let us emphasise that separability and squatness are logically unrelated, a point which is especially relevant in Theorem \ref{fall in}. The original (collared) Pr\"ufer surface is squat but not separable. On the other hand Nyikos has described (unpublished) a surface-with-boundary $N$ whose interior is $\mathbb R^2$ and whose boundary is $\mathbb L_+$. The doubled surface $2N$ is separable but not squat.

\begin{lemma}[Tube Lemma] \label{tube} {\rm(cf. \cite{Gauld})}
Suppose $L$ is a leaf of a dimension-one foliation $\mathcal F$ on a manifold $M^m$ and that $e:[0,2]\to L$ is an embedding. Then there is a foliated chart $(U,\varphi)$ such
that $e([0,2])\subset U$.
\end{lemma}

\begin{proof}
There is a partition $\{0=t_0<t_1<\dots<t_n=2\}$ of $[0,2]$
and foliated charts $(U_1,\varphi_1),\dots, (U_n,\varphi_n)$ so
that for each $i$, $e([t_{i-1},t_i])\subset U_i$. Thus using
induction on $i$ it suffices to show:
\begin{itemize}
\item if there are foliated charts $(U,\varphi)$ and $(V,\psi)$
such that $e([0,1])\subset U$ and $e([1,2])\subset V$ then
there is a foliated chart $(W,\chi)$ such that $e([0,2])\subset
W$.
\end{itemize}
We may assume that $\psi e(1)=(0,\dots,0)$. Let $C$ be a closed subset of $\mathbb R^m$ of the form $[a,b]\times\mathbb B^{m-1}$ containing $\psi e([1,2])$ in its interior, where $\mathbb B^{m-1}$ is the closed unit ball in $\mathbb R^{m-1}$. Let $\eta:[a,b]\to[a,b]$ be a homeomorphism fixing the end points and sending $0$ to the first coordinate of $\psi e(2)$. Let $\eta_t:[a,b]\to[a,b]$ be an isotopy fixing $\{a,b\}$ such that $\eta_0=\eta$ and $\eta_1$ is the identity. Define $\theta:C\to C$ by $\theta(x,y)=(\eta_{\|y\|}(x),y)$ for $(x,y)\in C$. See Figure \ref{TubeLemma}.

Let $W=\left( U-\psi^{-1}(C)
\right)\cup\left(\psi^{-1}\theta\psi(U\cap V)\right)$ and
define $\chi:W\to\mathbb R^m$ by
\[ \chi(\xi)=\left\{
\begin{array}{ccc} \varphi(\xi) & \mbox{if} & \xi\in
U-\psi^{-1}(C)\\
\varphi\psi^{-1}\theta^{-1}\psi(\xi) & \mbox{if} &
\xi\in\psi^{-1}\theta\psi(U\cap V). \end{array} \right.\]
Then $(W,\chi)$ is the required chart.
\end{proof}

\begin{figure}[h]
\centering
\scalebox{1.0}{\includegraphics{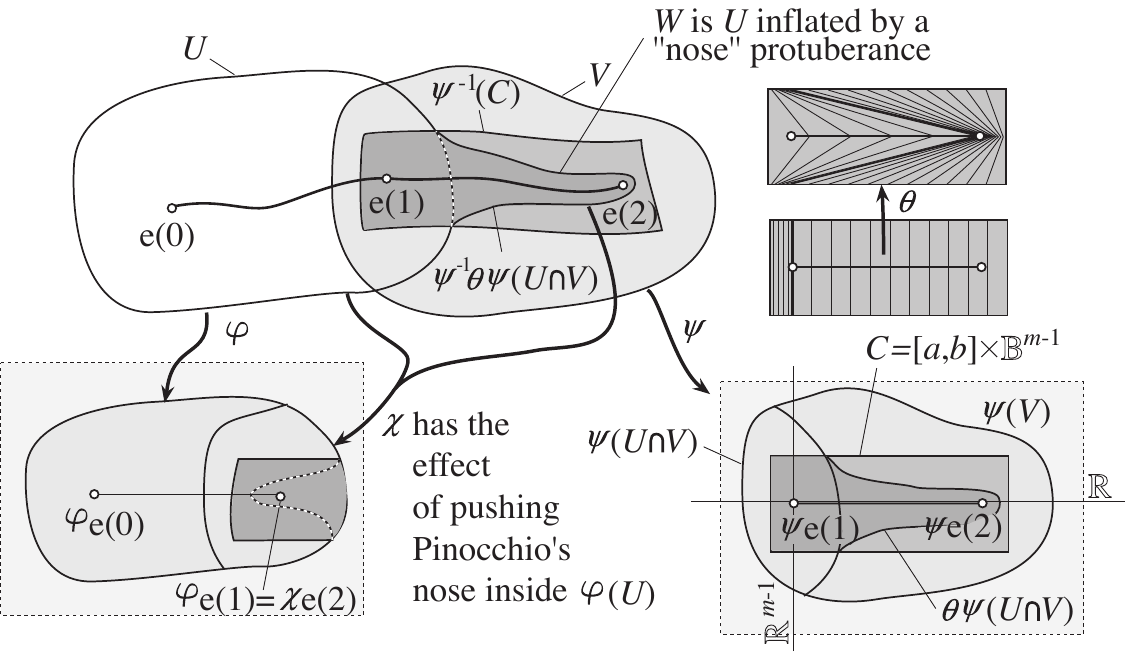}}
    \caption{\label{TubeLemma} Extending a foliated chart along a leaf}
\end{figure}

%Old version of the figure that didn't work with arXiv cut to........
\iffalse
\begin{figure}[h]
\centering
    \epsfig{figure=TubeLem.pdf,width=115mm}
    \caption{\label{TubeLemma} Extending a foliated chart along a leaf}
\end{figure}
\fi
%.........................here.

\begin{lemma}  \label{density}
Let $\mathcal F$ be a foliation of dimension $n$ on the product
$M\times N$ of a manifold $M$ with a connected $n$-manifold $N$.
Assume that there is a dense subset $D$ of $M$ such that for
each $d\in D$ the subset $\{d\}\times N$ is a leaf of
$\mathcal F$. Then $\mathcal F$ is the trivial product foliation with leaves of the form $\{x\}\times N$ for $x\in M$.
\end{lemma}

\begin{proof} Let $x\in M$: we show that $\{ x \} \times N$ is a leaf. Choose a sequence $\langle d_k \rangle$ from $D$ with $d_k\to x$. Suppose $y \in N$ and choose a
foliated chart $(U,\varphi)$ about $(x,y)$ with $\varphi(x,y)=0$ and open $O\subset N$ so that $y\in O$ and $\{x\}\times O\subset U$.
Because $\{ d_k \}\times N$ is a leaf for each $k$,
it follows that all points of $\varphi((\{ d_k \}\times N)\cap U)$ have the last $n$ coordinates the same. Moreover, as $(d_k, y) \to (x,y)$ and $\varphi(x,y)=0$ these coordinates must all go to $0$ as $k\to
\infty$. For each $z\in O$, as $(d_k,z)\to(x,z)$, then
$\varphi(x,z)$ has last coordinates all equal to
$0$. It follows that $\{x \}\times O$ lies in a single leaf.
As $y\in N$ was arbitrary and $N$ is connected, it follows that $\{x
\}\times N$ is a single leaf.
\end{proof}

% \section{Black holes.}
% \section{Black holes.}
% \section{Black holes.}
% \section{Black holes.}
 \section{Black holes} \label{Black holes}

In this section we use the concept of a squat manifold to analyse the asymptotic behaviour of a dimension-one foliation on a product manifold $M\times\mathbb L_+$, provided that $M$ is ``sufficiently small'' in the sense that it is both squat and separable. Basically squatness forces an individual ``long'' leaf to move vertically inside the product while separability enable us to extend this individual verticality to a collective one for the foliation.

To state the first result let us agree on some terminology:

\begin{defn}
Call a one-dimensional (Hausdorff) manifold \emph{long} if it
is non-metrisable (so by the classification it is either the
long ray or the long line), and \emph{short} otherwise.
\end{defn}

\begin{theorem}\label{fall in}
Suppose that $M$ is a squat, separable manifold and that $\mathcal F$ is a dimension-one foliation on $M\times\mathbb L_+$ having at least one long leaf. Then there is $\alpha\in\mathbb L_+$ so that $\mathcal F$ restricted to $M\times(\alpha,\omega_1)$ is the trivial product foliation by long rays. Say in this case that the foliation $\mathcal F$ is ultimately vertical.
\end{theorem}

\begin{proof*} The proof breaks into three steps.

\smallskip
\noindent {\bf Step 1.} {\it If $L$ is a long leaf of $\mathcal F$ then there are $x\in M$ and $\alpha\in {\mathbb L}_+$ such that $L \supset\{x\}\times[\alpha,\omega_1)$. (Say in this case that the foliation is vertical above the point $(x, \alpha)$).}

\smallskip
Let $i:\mathbb L_+\to L$ be an embedding. As $M$ is squat the $M$-coordinate of $i$ is eventually constant, say equal to $x$ after some $\beta \in{\mathbb L}_+$. Next the $\mathbb{L}_+$-coordinate of $i$ cannot be bounded for if it were then it would be contained in a homeomorph of $\mathbb R$, which is squat, so the second coordinate would be eventually constant, violating the injectivity of $i$. It follows that $i([\beta,\omega_1))=\{x\}\times [\alpha,\omega_1)$, where $\alpha$ is the ${\mathbb L}_+$-coordinate of $i(\beta)$. This establishes Step 1.

\medskip
\noindent {\bf Step 2.} {\it Let
\[ A=\{x\in M:\mbox{there is } \alpha\in\mathbb{L}_+ \mbox{ so that } \{x\}\times[\alpha,\omega_1) \mbox{ lies in a single leaf of } \mathcal F\}.\]
Then we claim that $A=M$.}

\smallskip
Since $M$ is connected, it suffices to show (i)
$A\not=\varnothing$; (ii) $A$ is open; (iii) $A$ is closed.
\begin{itemize}
\item[(i)] $A\not=\varnothing$. This follows from the
assumption that $\mathcal F$ has at least one long leaf and
Step 1.

\item[(ii)] $A$ is open.

Let $x\in A$, so there is an $\alpha \in \mathbb L_+$ so that
the foliation $\mathcal F$ is vertical above the point
$(x,\alpha)$. Since $x$ is a point in a manifold $M$ we can
fix a countable fundamental system of neighbourhoods
$(V_n)_{n\in \mathbb N}$. By applying Lemma \ref{tube} to the
arc $\{x\}\times[\alpha,\beta]$ for varying $\beta\in \omega_1$
greater than $\alpha$, we see that for each such $\beta$ there
is an $n\in {\mathbb N}$ such that ``{\it every leaf through
$V_n \times \{\alpha\}$ crosses $M\times \{\beta\}$}''. Call this last
(italicised) statement $S(n,\beta)$ and let $S=\{ (n,\beta)\in
\mathbb N\times\omega_1 :  S(n,\beta) {\rm \;is \;true}\}$.
By the argument above the set $S$ is uncountable, hence there is an $n\in {\mathbb N}$ so that $S\cap (\{n\}\times \omega_1)$ is uncountable. This means that each leaf
through $V_n\times \{\alpha\}$ crosses $M\times \{\beta\}$ for
uncountably many $\beta > \alpha$. In particular each such
leaf is long.

By squatness of the base $M$, each long leaf of $\mathcal F$
stabilises and becomes {\it purely vertical} above some height
$\alpha\in {\mathbb L}_+$, i.e. the leaf intersects $M\times[\alpha,\omega_1)$ in one or two vertical segments (depending on whether the long leaf under inspection is a long ray or a long line).

Take now $D$ a countable dense subset of $V_n\times \{\alpha\}$.
Each leaf $L_d$ through the point $d\in D$ is long,
and so by the previous discussion there is a height
$\alpha_d\in {\mathbb L}_+$ above which $L_d$ is
purely vertical. Consider $\alpha_D=\sup_{d\in D} \alpha_d \in
{\mathbb L}_+$. Apply Lemma \ref{tube} to
$\{x\}\times[\alpha,\alpha_D]$ to produce a foliated chart $(U, \varphi)$
with $U\supset \{x\}\times[\alpha,\alpha_D]$. Looking through the
chart one obtains a pair of neighbourhoods $N$, $N'$ of
$\varphi(x, \alpha)$ in $\varphi((V_n\times \{\alpha\}) \cap U)$
respectively of $\varphi(x, \alpha_D)$ in $\varphi((M\times
\{\alpha_D\}) \cap U)$ related by a homeomorphism $h: N \to N'$
which is just propagation along the vertical straight lines
(see Figure \ref{FoliationTube}). Let $\Delta:=\varphi(D\cap
U) \cap N$: by construction the foliation $\mathcal F$ is
vertical above $\varphi^{-1}(h(\Delta))$. Since
$\varphi^{-1}(h(\Delta))$ is dense in $\varphi^{-1}(N')$ it
follows from Lemma \ref{density} that $\mathcal F$ is vertical
above the neighbourhood $\varphi^{-1}(N')$, hence the
$M$-projection of $\varphi^{-1}(N')$ is a neighbourhood of $x$
contained in $A$.

\begin{figure}[h]
\centering
\scalebox{1.1}{\includegraphics{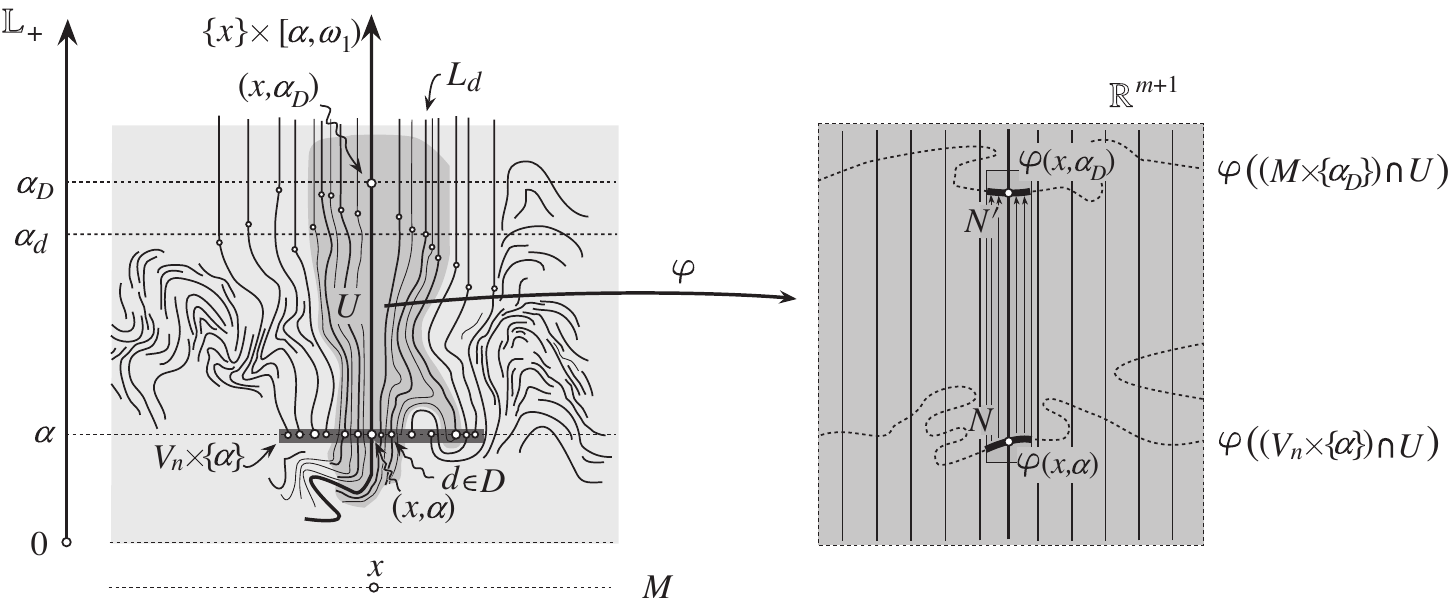}}
    \caption{\label{FoliationTube} Applying the tube lemma around a vertical leaf}
\end{figure}

%Old version of the figure that didn't work with arXiv cut to........
\iffalse
\begin{figure}[h]
\centering
    \epsfig{figure=FolTube.pdf,width=155mm}
    \caption{\label{FoliationTube} Applying the tube lemma around a vertical leaf}
\end{figure}
\fi
%....................here.

\item[(iii)] $A$ is closed. Since $M$ is first countable, it suffices to show that $A$ is sequentially closed. Let $\langle x_n\rangle$ be a sequence in $A$ converging to $x$. For each $n$ there is $\alpha_n\in\mathbb L_+$ so that $\mathcal F$ is vertical above $(x_n,\alpha_n)$. We may assume that the sequence $\langle\alpha_n\rangle$ is increasing. Let $\alpha=\lim_{n\to\infty}\alpha_n\in\mathbb L_+$. Using foliated charts centred at various points $(x,\beta)$ with $\beta\ge\alpha$, it is routine to check that the foliation is vertical  above the point $(x,\alpha)$, and hence that $x\in A$.

\end{itemize}

\medskip
\noindent {\bf Step 3.} {\it $\mathcal F$ is ultimately vertical.}

\smallskip
Since $M$ is separable, it admits a countable dense subset $D$. For each $d\in D$ there is by Step 2 an $\alpha_d\in {\mathbb L}_+$ such that $\mathcal F$ is vertical above the point $(d,\alpha_d)$. Take $\alpha=\sup_{d\in D} \alpha_d$: then $\mathcal F$ is vertical above the subset $D \times \{\alpha\}$. By Lemma \ref{density}, one concludes that $\mathcal F$ is vertical above $M \times \{\alpha\}$. \hfill\rule{0.5em}{0.5em}

\end{proof*}

\begin{prop}\label{saturate}
Suppose that $M$ is a manifold with a decomposition of the
form $M=\cup_{\alpha\in\omega_1}U_\alpha$, where each
$U_\alpha$ is separable and open, $\overline{U_\alpha}\subset
U_\beta$ whenever $\alpha<\beta$, and
$U_\lambda=\cup_{\alpha<\lambda}U_\alpha$ whenever $\lambda$
is a limit ordinal. Suppose that $\mathcal F$ is a foliation
on $M$ for which all leaves are metrisable. Then
$C=\{\alpha\in\omega_1:U_\alpha \mbox{ is saturated by } \mathcal F\}$
is a closed unbounded subset of $\omega_1$. In particular, for
each $\alpha\in C$, the set $\overline{U_\alpha}- U_\alpha$ is
saturated by $\mathcal F$.
\end{prop}

\begin{proof} Recall that a connected metrisable
manifold is Lindel\"of, hence each leaf of $\mathcal F$ is
contained in some $U_\alpha$ for $\alpha \in \omega_1$ (the
leaf is then said to be \emph{bounded by $\alpha$}).
We show that $C$ is unbounded. Construct an increasing
sequence $\langle\alpha_n\rangle$ in $\omega_1$ as follows.
Let $\alpha_0\in\mathbb \omega_1$ be arbitrary. Now suppose given $\alpha_n$.

Let $D_\alpha\subset U_\alpha$ be a countable dense subset and
consider the leaves of $\mathcal F$ which pass through points
of $D_{\alpha_n}$. Because each leaf is bounded, collectively they all are
bounded, say by $\alpha_{n+1}>\alpha_n$.
We claim that $L\subset \overline{U_{\alpha_{n+1}}}$ for each leaf
$L$ of $\mathcal F$ for which $L\cap U_{\alpha_n}\not=\varnothing$.

Suppose that $L$ is a leaf with $L\cap
U_{\alpha_n}\not=\varnothing$ and let $e:[0,1]\to L$ be any
embedding so that $e(0)\in U_{\alpha_n}$. To show that
$L\subset \overline{U_{\alpha_{n+1}}}$ it suffices to show
that $e(1)\in \overline{U_{\alpha_{n+1}}}$, because the
arcwise-connexity of $L$ allows the end-point $e(1)$ to reach
any point of $L$. By Lemma \ref{tube} there is a foliated chart
$(U,\varphi)$ so that $e([0,1])\subset U$.
Choose $\langle x_n\rangle$ a sequence in $D_{\alpha_n}$ converging to $e(0)$. Since $U$ is open the sequence eventually lands in $U$, and so by means of the chart we easily construct a sequence $\langle y_n\rangle$ converging to $e(1)$ by setting $\varphi^{-1}(L_n\cap H)=\{y_n\}$, where $L_n$ is the straight line through  $\varphi(x_n)$ and $H$ is the orthogonal hyperplane through $\varphi(e(1))$. By construction $y_n$ belongs to the same leaf as $x_n$, so each $y_n$ is in $U_{\alpha_{n+1}}$, and therefore the limit $e(1)$ belongs to $\overline{U_{\alpha_{n+1}}}$.

Now let $\alpha=\lim\alpha_n$. Then $\alpha\in C$, because if $L$ is any leaf meeting $U_\alpha$ then $L$ meets $U_{\alpha_n}$ for some $n$ (since $\alpha$ is a limit ordinal) and hence lies in $\overline{U_{\alpha_{n+1}}}\subset U_\alpha$, so $U_\alpha$ is saturated.

That $C$ is closed follows from Criterion \ref{closed in omega_1} and the fact that a union of $\mathcal F$-saturated subsets is saturated. \end{proof}

%{Black hole consequences}
%{Black hole consequences}
%{Black hole consequences}
%{Black hole consequences}
\section{Black hole consequences}\label{Black hole consequences}

In this section we complete our analysis of foliations on the simplest long pipe, $\mathbb S^1\times\mathbb  L_+$. We then exhibit a family of surfaces, almost none of which admit foliations but when we delete a point then the punctured surface admits a foliation. This is followed by a look at dimension-one foliations on $\mathbb  S^2\times\mathbb  L_+$.
\medskip

{\bf Proof of Theorem \ref{black hole}.} If there are only short leaves then by applying Proposition \ref{saturate} to $\mathbb S^1\times\mathbb L_+=\cup_{\alpha\in\omega_1}\mathbb S^1\times(0,\alpha)$ the situation described in (i) holds. (Notice that in Proposition \ref{saturate} we only obtained closedness of $C$ when it was restricted to $\omega_1$, but routine arguments will verify that $C$ is closed in $\mathbb L_+$.) On the other hand if there is a long leaf then situation (ii) follows from Theorem \ref{fall in}.
\hfill\rule{0.5em}{0.5em}

\begin{rem} \label{foliating cylinders}
{\rm In situation (ii) of Theorem \ref{black hole} it is possible to have a bounded collection of circular leaves running around the cylinder. The situation described in (i) is ``sharp'' in the sense that one cannot expect all leaves to be ultimately circular. Indeed first consider the \emph{Kneser foliation} on $\mathbb S^1\times[0,1]$, namely the unique foliation without circular leaves except the two boundaries which moreover is transverse to the foliation by intervals. Transfinite gluing of an $\omega_1$-sequence of such foliated annuli produces a foliation on $\mathbb S^1 \times {\mathbb L}_+$ such that the set $C$ is exactly $\omega_1$. At the opposite extreme the \emph{Reeb foliation} on the annulus $\mathbb S^1\times [0,1]$ develops ``singularities'' when reaching a limit ordinal. See Figure \ref{KneserReebFoliation}.}
\end{rem}

\begin{figure}[h]
\centering
\scalebox{1.0}{\includegraphics{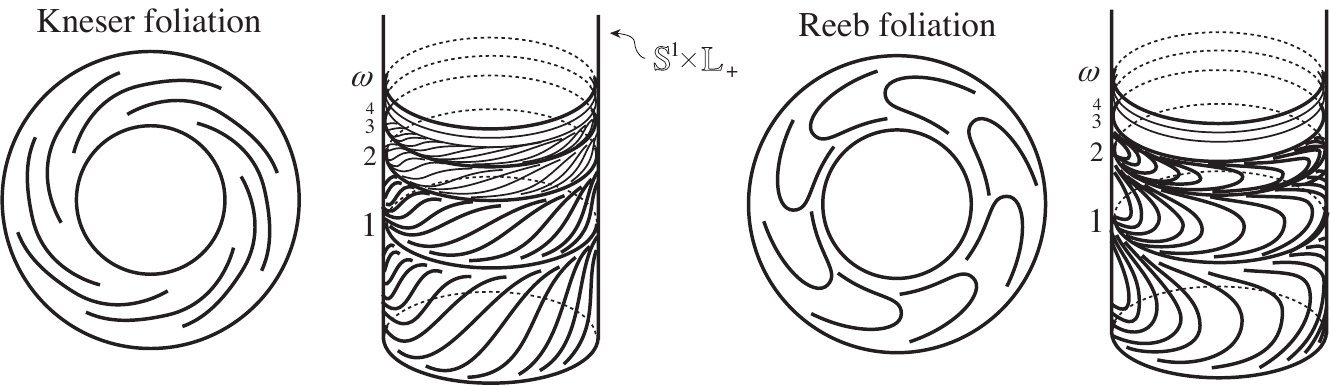}}
    \caption{\label{KneserReebFoliation} Transfinite gluing: possible with the Kneser foliation but impossible with the Reeb foliation}
\end{figure}

%Old version of the figure that didn't work with arXiv cut to........
\iffalse
\begin{figure}[h]
\centering
    \epsfig{figure=KnReFol.pdf,width=135mm}
    \caption{\label{KneserReebFoliation} Transfinite gluing: possible with the Kneser foliation but impossible with the Reeb foliation}
\end{figure}
\fi
%.............................here.

As an application of Theorem \ref{black hole} we now describe a family of open surfaces supporting no foliations. Start with $\Sigma_{g}$ a genus $g$ (orientable) closed surface. Cut out $n$ pairwise disjoint (open) disks to obtain $\Sigma_{g,n}$ a genus $g$ surface with $n$ boundary components, and glue back $n$ long cylinders $\mathbb S^1\times {\mathbb L}_{\ge 0}$ (see Figure \ref{Lambdagn}). The resulting surface ${\Lambda}_{g,n}$ could be termed the genus $g$ surface with $n$ {\it black holes}.

\begin{figure}[h]
\hspace{-0.5cm}
\begin{minipage}[b]{0.5\linewidth} % A minipage that covers half the page
\centering
\scalebox{0.9}{\includegraphics{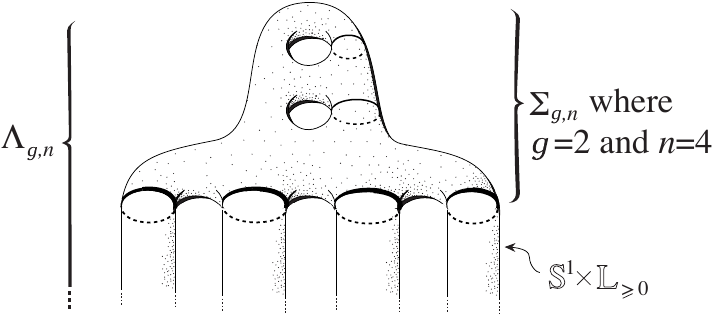}}
\caption{\label{Lambdagn} The genus $g$ surface
with $n$ black holes}
\end{minipage}
\hspace{-0cm} % To get a little bit of space between the figures
\begin{minipage}[b]{0.5\linewidth}
\centering
\scalebox{1.0}{\includegraphics{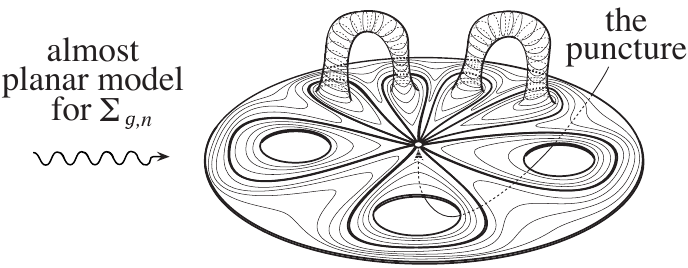}}
\caption{\label{FoliatedLambdagn} Foliation on the punctured surface of genus $g$ with $n$ boundary components}
\end{minipage}
\end{figure}

%Old version of the figure that didn't work with arXiv cut to........
\iffalse
\begin{figure}[h]
\hspace{-0.5cm}
\begin{minipage}[b]{0.5\linewidth} % A minipage that covers half the page
\centering
    \epsfig{figure=Lambdag.pdf,width=70mm}
\caption{\label{Lambdagn} The genus $g$ surface
with $n$ black holes}
\end{minipage}
\hspace{-0cm} % To get a little bit of space between the figures
\begin{minipage}[b]{0.5\linewidth}
\centering
    \epsfig{figure=FolLamb.pdf,width=70mm}
\caption{\label{FoliatedLambdagn} Foliation on the punctured surface of genus $g$ with $n$ boundary components}
\end{minipage}
\end{figure}
\fi
%...........................here.

\medskip
{\bf Proof of Corollary \ref{non-foliable surfaces}.} Assume that ${\Lambda}_{g,n}$ has a foliation (of dimension one). The complement ${\Lambda}_{g,n}-\Sigma_{g,n}$ splits into $n$ tubes $\mathbb S^1\times{\mathbb L}_+$, each equipped with an induced foliation. According to Theorem \ref{black hole} there is in each of those tubes a circle either tangent or transverse to the foliation. Cutting the surface ${\Lambda}_{g,n}$ along those $n$ circles and discarding the non-metrisable components leads to a surface-with-boundary homeomorphic to $\Sigma_{g,n}$ with a foliation well behaved along the boundary. Therefore it can be doubled and yields a foliation on the double $2\Sigma_{g,n}$ which is a surface of genus $2g+(n-1)$. This is possible only if $2g+(n-1)=1$, which has only $2$ solutions $(g,n)=(0,2)$ and $(1,0)$, corresponding to the two exceptional surfaces. \hfill\rule{0.5em}{0.5em}

\smallskip
{\it Note.} We used the classical fact that a closed surface of genus $g$ carries a dimension-one $C^0$-foliation only if $g=1$. In Lemma~\ref{Euler} in the Appendix below, we recall how this follows from the Lefschetz fixed-point theorem applied to a dyadic `cascadization' of Whitney's flow generated by the foliation.

\medskip
Let us observe the following:

\begin{prop} If one performs a single puncture in any of the surfaces
${\Lambda}_{g,n}$ then it has a foliation. \end{prop}

\begin{proof}
By sliding the boundary circles of $\Sigma_{g,n}$ along the tubes if necessary, one can assume the puncture to be located in the interior of the compact `nucleus'
$\Sigma_{g,n}$ of the surface ${\Lambda}_{g,n}$. It is
easy to see that $\Sigma_{g,n}-$(an interior point) carries a
foliation in which the boundary components are leaves, since the surface $\Sigma_{g,n}$ admits an almost planar model as a disk with $n-1$ holes to which $g$
handles are attached (see Figure \ref{FoliatedLambdagn}).
This foliation can be extended to ${\Lambda}_{g,n}-\ast$ by foliating the tubes with circles.
\end{proof}

\noindent {\bf Question.} Is there a surface which does not admit a foliation even after
puncturing?

\medskip
When the base manifold $M$ has no foliation of dimension $1$, then the manifold $M\times {\mathbb L}_+$ behaves more cannibalistically forcing the leaves to fall into the hole in a purely vertical way:

\begin{cor}\label{cannibal}
Let $\mathcal F$ be a dimension-one foliation on $\mathbb S^2\times\mathbb L_+$. Then there is $\alpha\in\mathbb L_+$ so that $\mathcal F$ restricted to $\mathbb S^2\times(\alpha,\omega_1)$ is the trivial product foliation by long rays. More generally the same conclusion holds when $\mathbb S^2$ is replaced by any closed (topological) manifold with non vanishing Euler characteristic.
\end{cor}

\begin{proof}
According to Theorem \ref{fall in} it suffices to show that there is a long leaf. If not, then by  Proposition~\ref{saturate} applied to $\cup_{\alpha\in\omega_1}\mathbb S^2\times(0,\alpha)$ there would have to be $\alpha\in\omega_1$ such that $\mathcal F$ restricts to a dimension-one foliation on $\mathbb S^2\times\{\alpha\}$. However $\mathbb S^2$ carries no dimension-one foliation (Lemma~\ref{Euler}). The more general statement follows by the same argument using the classical fact that a closed manifold $M$ with $\chi(M)\neq 0$ has no dimension-one foliation (see Lemma \ref{Euler}).
\end{proof}

It is not the case that up to equivalence there is a unique dimension-one foliation on $\mathbb S^2 \times {\mathbb L}_+$. For example one can perturb the ``radial'' foliation on $\mathbb S^2\times {\mathbb L}_+$ by long rays, along an ellipsoid of revolution touching the north and south poles of $\mathbb S^2 \times \{ 0\}$ equipped with a longitudinal motion from the south to the north pole (see Figure \ref{NonTrivialRadialFoliation}). The resulting foliation is clearly not equivalent to the trivial radial foliation since it has many short leaves (inside the ellipsoid, each being homeomorphic to ${\mathbb R}$).

\begin{figure}[h]
\centering
\scalebox{1.0}{\includegraphics{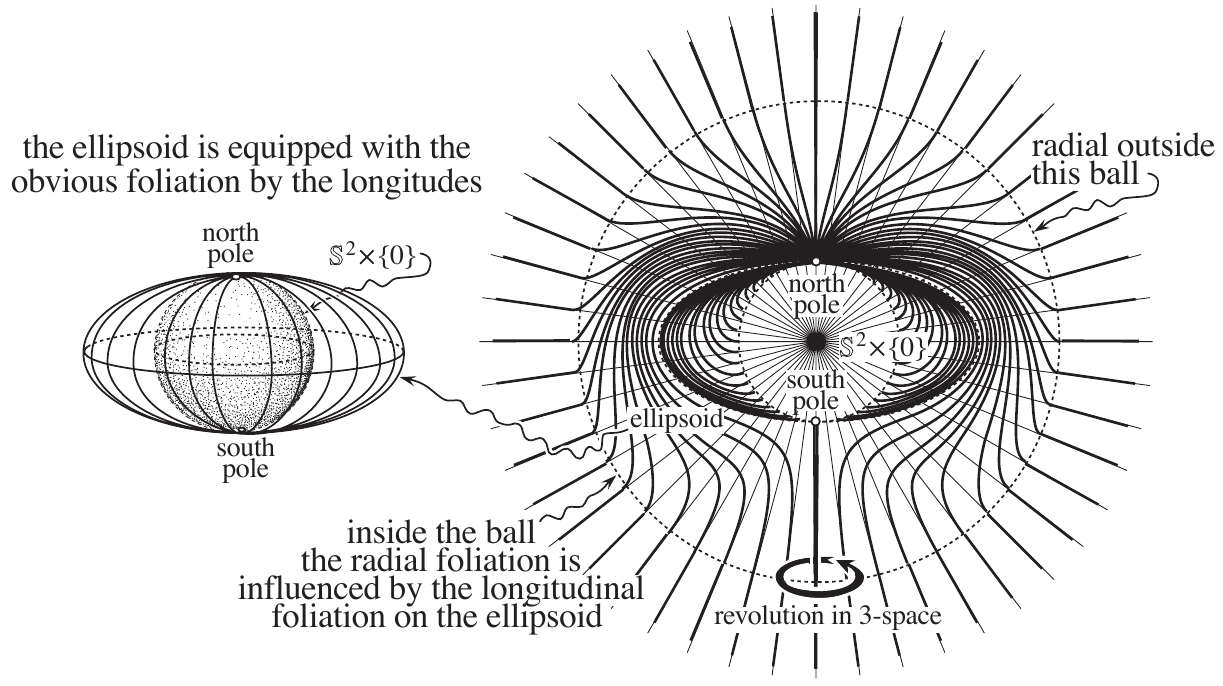}}
    \caption{\label{NonTrivialRadialFoliation} A non-trivial
    asymptotically radial foliation on ${\mathbb S}^2\times\mathbb L_+$}
\end{figure}

%Old version of the figure that didn't work with arXiv cut to........
\iffalse
\begin{figure}[h]
\centering
    \epsfig{figure=NTRadFo.pdf,width=130mm}
    \caption{\label{NonTrivialRadialFoliation} A non-trivial
    asymptotically radial foliation on ${\mathbb S}^2\times\mathbb L_+$}
\end{figure}
\fi
%......................here.

If we seek a foliation (still on $\mathbb S^2\times {\mathbb L}_+$) realising $\mathbb S^1$ as a leaf, we need only to alter slightly the construction above by equipping the ellipsoid with the foliation by latitudinal circles. This gives a foliation on $\mathbb S^2\times {\mathbb L}_+$ exhibiting all topological types of one-manifold as leaves except for the long line. (Inside the ellipsoid the leaves are lines $\mathbb R$ spiraling
asymptotically to a latitude, on the ellipsoid we have
circular leaves, and outside the ellipsoid we see long rays
spiraling in general towards a latitudinal circle.)

Actually the long line cannot be realised as a leaf of a foliation on $\mathbb S^2 \times {\mathbb L}_+$. Indeed by ultimate verticality (Corollary \ref{cannibal}) there would be an induced foliation on some truncation $\mathbb S^2\times(0,\alpha]$ transverse to the boundary. Since $\mathbb S^2\times(0,\alpha]$ is simply connected the foliation is orientable and so associated to a flow by a classical result of Whitney~\cite{Whitney33}. Since the boundary is connected a standard ``closed-open set argument'' shows that the flow can be assumed to be pointing inward everywhere along the boundary. Hence there cannot be a long-line leaf, because its restriction to $\mathbb S^2\times(0,\alpha]$ would be an arc joining two boundary-points, which when oriented would point inward at one point and outward at the other, contradicting the global inwardness of the flow.

Obviously the same argument applies to any ultimately vertical foliation on $M\times {\mathbb L}_+$, because one can always lift the foliation to the universal cover $\widetilde{M}\times{\mathbb L}_+$ to ensure the orientability of the foliation. Since long-line leaves are conserved when lifting to a cover, we obtain the following addendum to Theorem \ref{fall in}:

\begin{prop}
Under the assumptions of Theorem \ref{fall in}, no leaf can be a long line.
\end{prop}

\begin{cor} \label{longglass}
The $n$-dimensional long glass $\Lambda^n={\Bbb B}^n\cup_{\partial}({\Bbb S}^{n-1}\times {\Bbb L}_{\ge 0})$ has no dimension-one foliation for each $n\ge 2$.
\end{cor}
\begin{proof} By contradiction assume ${\cal F}$ is such a
foliation. We restrict ${\cal F}$ to the ``pipe'' $P:={\Bbb
S}^{n-1}\times {\Bbb L}_+$.
If the restricted foliation ${\cal F}_P$ contains a long leaf,
then ${\cal F}_P$ is ultimately vertical by Theorem~\ref{fall
in}, say after some $\alpha\in {\Bbb L}_+$. Restricting to the
compact subregion ${\Bbb B}^n\cup_{\partial}({\Bbb
S}^{n-1}\times [0,\alpha])$ yields a foliation on (a
homeomorph of) the ball ${\Bbb B}^n$ transverse to its
boundary. Such a foliation would imply a flow
(Whitney~\cite{Whitney33}) entrant throughout the connected
boundary sphere, contradicting the Brouwer fixed point
theorem.
Otherwise ${\cal F}_P$ contains only short leaves. Then
according to Proposition~\ref{saturate} there is $\alpha\in
\omega_1$ such that ${\Bbb S}^{n-1}\times (0, \alpha)$ is
saturated by ${\cal F}_P$. This implies that ${\Bbb
B}^n\cup_{\partial}({\Bbb S}^{n-1}\times [0,\alpha])$ is
saturated by ${\cal F}$, yielding again an impossible
foliation on the ball.
\end{proof}

The most general statement that can be deduced conjointly from Theorem \ref{fall in} and Proposition \ref{saturate} is:

\begin{cor}
Let $\mathcal F$ be a dimension-one foliation on $M\times\mathbb
L_+$, where $M$ is a manifold that is squat, separable and
without dimension-one foliations. Then there is
$\alpha\in\mathbb L_+$ so that $\mathcal F$ restricted to
$M\times(\alpha,\omega_1)$ is the trivial product foliation by
long rays.
\end{cor}

\begin{rem}
{\rm Unfortunately we do not know any example of such an $M$ which is not compact! We suspect that a strong candidate is a mixed Pr\"ufer--Moore surface, i.e. we start from the Pr\"ufer surface with boundary $P$, auto-glue by $x\sim -x$ some of the boundaries of $P$ and then take the double. By the way it would be interesting to find an example of a non-foliable surface which is separable (without being compact).}
\end{rem}

We now make some remarks deploring our poor knowledge of codimension-one foliations on $\mathbb S^2\times {\mathbb L}_+$.

\begin{ques} \label{foliate S^2xL_+}
Is it true that any codimension-one foliation on $\mathbb S^2\times\mathbb L_+$ has a spherical leaf (i.e. homeomorphic to $\mathbb S^2$.)
\end{ques}

The answer is {\it yes} under the extra assumption that all leaves are bounded (apply Proposition \ref{saturate}).

If the answer is yes then it would follow from Reeb's stability theorem \cite{Reeb} that the foliation ultimately consists only of spherical leaves. Nevertheless there may well be a toral leaf as illustrated in Figure \ref{ToralLeaf}.

\begin{figure}[h]
\centering
\scalebox{0.9}{\includegraphics{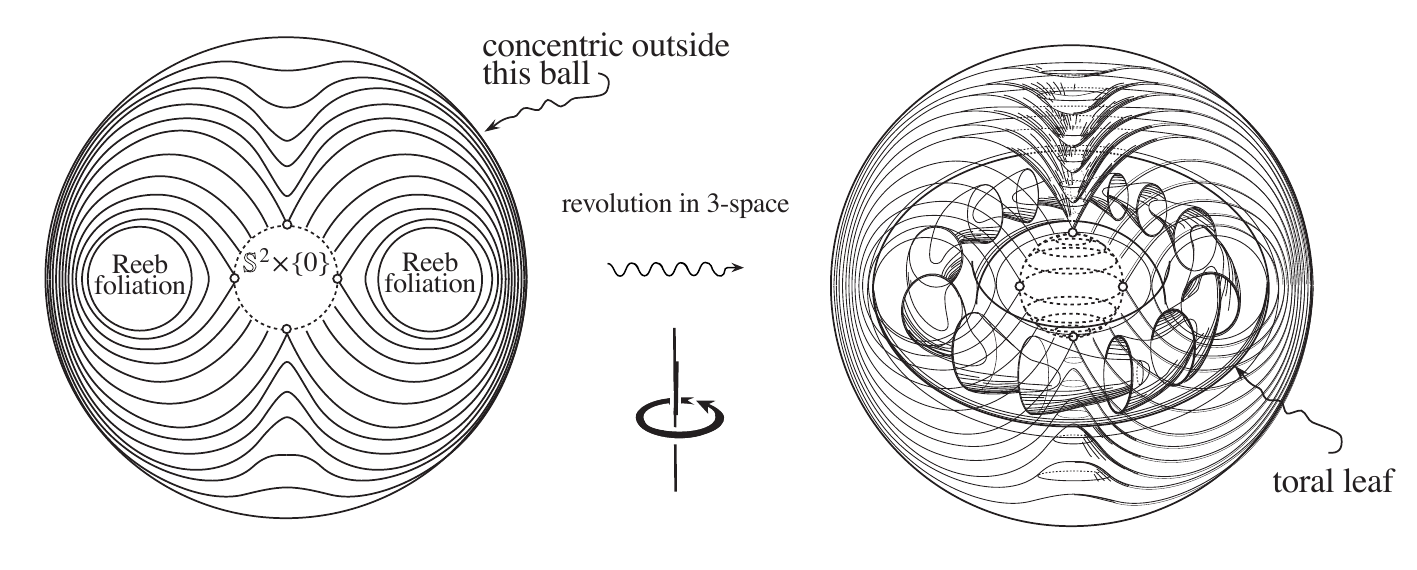}}
    \caption{\label{ToralLeaf} A foliation on
    ${\mathbb S}^2\times\mathbb L_+$ with a toral leaf}
\end{figure}

%Old version of the figure that didn't work with arXiv cut to........
\iffalse
\begin{figure}[h]
\centering
    \epsfig{figure=TorLeaf.pdf,width=130mm}
    \caption{\label{ToralLeaf} A foliation on
    ${\mathbb S}^2\times\mathbb L_+$ with a toral leaf}
\end{figure}
\fi
%..........................here.

Are compact leaves of higher genus $g\ge 2$ precluded? As weak positive evidence we have the following:

\begin{prop} In case of a positive answer to Question \ref{foliate S^2xL_+}, there cannot be leaves of genus $\ge 2$.
\end{prop}

\begin{proof} Clearly using homology theory any closed leaf $F$ splits its complement in $\mathbb S^2\times {\mathbb L}_+$ into two components, exactly one of which is metrisable. Call the (closure of the) latter the {\it inside of the leaf} $F$, denoted by ${\rm ins}(F)$. We distinguish two cases depending on whether the inside is open or not.

In the first case the region between the spherical leaf $S$ and
$F \cong \Sigma_g$ would be compact and this is impossible by
Reeb's stability. The precise meaning of ``the region between''
is the symmetrical difference ${\rm ins}(S) \bigtriangleup{\rm
ins}(F)$ of the insides. The compactness comes from the fact
that ${\rm ins} (S)$ is open as well, because otherwise
according to the Alexander--Schoenflies theorem ${\rm ins}(S)$
would be a 3-ball, but the latter has no foliation.

In the second case ${\rm ins}(F)=W$ would be a compact
3-manifold with boundary. Now doubling $W$, we have $\chi(2W)=2\chi(W)-\chi(\partial W)$ and the double $2W$ has zero Euler characteristic by
Poincar\'e duality so $\chi(W)=1-g$. Now consider
${\mathcal F}_{W}$ the induced foliation on $W$. If it is
transversely orientable then we can find a flow $(f_t)$ on
$W$ pointing inward (everywhere along the connected
boundary). This would violate the Lefschetz fixed-point
theorem. [Indeed if $t_n={1\over 2^n}$ denotes the dyadic time
then $K_n={\rm Fix}(f_{t_n})$ would be non-void by Lefschetz,
so in the nested intersection $\cap_{n=1}^{\infty} K_n$ one
finds a point at rest for all time of the flow. This
contradicts the fact that the orbits of the flow are exactly
the leaves of the foliation, none of which can
collapse to a single point.] The general case follows by noting that
the foliation being the restriction of a foliation defined on
the simply-connected manifold $\mathbb S^2\times(0,\alpha)$ for some
sufficiently large $\alpha$ is automatically transversely
orientable.
\end{proof}

%\section{Foliating $\mathbb L^2-\{\rm pt\}$.}
%\section{Foliating $\mathbb L^2-\{\rm pt\}$.}
%\section{Foliating $\mathbb L^2-\{\rm pt\}$.}
%\section{Foliating $\mathbb L^2-\{\rm pt\}$.}
\section{Foliating large subsets of $\mathbb L^2$} \label{Foliating L2-pt}

We begin this section by determining the asymptotic behaviour of dimension-one foliations on the long plane, possibly punctured by the removal of a compact subset. Up to certain rigid motions, only six possible pictures will emerge, as depicted in Figure \ref{foliateL^2-pt}. The key idea in detecting these six asymptotic structures is to cut $\mathbb L^2$ not along the axes but along the two diagonals. Doing so yields quadrants which, while not being themselves products, can be filled by strips such as $(-\alpha, \alpha)\times(\alpha,\omega_1)$ having a ``squat-long'' product decomposition to subordinate their foliation theory to the general methods of Section \ref{Black holes}. Note the special case where the puncture is just a single point: in that case we have foliated a second long pipe.

As shown in Figure \ref{foliateL^2-pt} all except some regions have prescribed foliations. In the second part of the section we investigate how these regions may be foliated. We are able to conclude that $\mathbb L^2$ has two foliations up to homeomorphism and six up to isotopy.

\begin{prop}\label{quadrant constraint} 
Let $K$ be a compact subset of $\mathbb L^2$ and suppose that $\mathcal F$ is a dimension-one foliation on $\mathbb L^2-K$. Within the quadrant $Q=\{(x,y)\in\mathbb L^2:-y<x<y\}$ exactly one of the following must hold:
\begin{itemize}
\item $\{\alpha\in\mathbb L_+:Q\cap\left([-\alpha,\alpha]\times\{\alpha\}\right)\mbox{ is part of a leaf of } \mathcal F \}$ is  a closed unbounded subset of $\mathbb L_+$;
\item $\{\alpha\in\mathbb L_+: \{x\}\times[\alpha,\omega_1)\mbox{ is part of a leaf of } \mathcal F  \mbox{ for each } x\in[-\alpha,\alpha]\}$ is  a closed unbounded subset of $\mathbb L_+$.
\end{itemize}
\end{prop}
\begin{proof}
Either no leaf of $\mathcal F$ meets $Q$ in an unbounded set, which we show leads to the first option, or there is a leaf whose intersection with $Q$ is unbounded, which will lead to the second option. We may assume that $K\subset [-\frac{1}{2},\frac{1}{2}]$.

The first case follows by applying Proposition \ref{saturate} to $Q=\cup_{\alpha>0}(Q\cap (\mathbb L\times(0,\alpha)))$, noting that bounded leaves are Lindel\"of and hence metrisable.

Unboundedness in the second case may be deduced from Theorem  \ref{fall in} as follows. We know there is at least one unbounded leaf in $Q$: choose one such leaf and denote it by $L$. Let $\alpha_0\in\omega_1$: we construct an increasing sequence $\langle\alpha_n\rangle$ as follows. Any leaf unbounded in $Q$ must have bounded first coordinate (otherwise an easy argument shows that it meets the boundary of $Q$ in a closed unbounded set) and hence, by Lemma \ref{Mathieu}, eventually has constant first coordinate. Hence we may assume that $\alpha_0$ is big enough that $L\subset Q\cap((-\alpha_0,\alpha_0)\times\mathbb L_+)$. Given $\alpha_n$, the manifold $(-\alpha_n,\alpha_n)$ satisfies the hypotheses of the manifold $M$ in Theorem \ref{fall in}. Hence there is $\alpha_{n+1}>\alpha_n$ such that $\mathcal F$ restricted to $(-\alpha_n,\alpha_n)\times(\alpha_{n+1},\omega_1)$ is the trivial product foliation by long rays. Letting $\alpha=\lim\alpha_n$ we find that $\alpha>\alpha_0$ and $\{x\}\times[\alpha,\omega_1)$ is part of a leaf of $\mathcal F$ for each $x\in(-\alpha,\alpha)$. It follows from Lemma \ref{density} that $\{x\}\times[\alpha,\omega_1)$ is part of a leaf of $\mathcal F$ for each $x\in[-\alpha,\alpha]$.

It is routine, using convergent sequences, to show that the set in the second case is closed.
\end{proof}

\begin{rem} \label{semi-diagonals not leaves}
{\rm It follows from Proposition \ref{quadrant constraint} that the semi-diagonals $\{(\pm x,\pm x) : x\in\mathbb L_+\}$ cannot be leaves of any dimension-one foliation on $\mathbb L^2-K$.} 
\end{rem}

Figure \ref{foliateL^2-pt} illustrates the six cases respectively of the following theorem. Arrowheads indicate that the leaf is long in the direction of the arrow. The foliation may be extended arbitrarily when not explicitly prescribed.

\begin{theorem} \label{foliate long plane}
Let $K$ be a compact subset of $\mathbb L^2$ and suppose that $\mathcal F$ is a dimension-one foliation on $\mathbb L^2-K$. Then there is a closed unbounded set $C\subset\mathbb L_+$ so that, up to rotation of the axes, appropriate leaves of $\mathcal F$ must take one of the following forms for each $\alpha\in C$:
\begin{enumerate}
\item $(\{\pm\alpha\}\times[-\alpha,\alpha])\cup([-\alpha,\alpha]\times\{\pm\alpha\})$;
\item $(\{\pm\alpha\}\times(-\omega_1,\alpha])\cup([-\alpha,\alpha]\times\{\alpha\})$;
\item $(\{\alpha\}\times(-\omega_1,\alpha])\cup((-\omega_1,\alpha]\times\{\alpha\})$;
\item $(-\omega_1,\omega_1)\times\{\alpha\}$ and $(-\omega_1,\omega_1)\times\{-\alpha\}$;
\item $(-\omega_1,\omega_1)\times\{\alpha\}$,  $((-\omega_1,-\alpha]\times\{-\alpha\})\cup(\{-\alpha\}\times(-\omega_1,-\alpha])$ and $([\alpha,\omega_1)\times\{-\alpha\})\cup(\{\alpha\}\times(-\omega_1,-\alpha])$;
\item $((-\omega_1,-\alpha]\times\{\alpha\})\cup(\{-\alpha\}\times[\alpha,\omega_1))$, $([\alpha,\omega_1)\times\{\alpha\})\cup(\{\alpha\}\times[\alpha,\omega_1))$,  $((-\omega_1,-\alpha]\times\{-\alpha\})\cup(\{-\alpha\}\times(-\omega_1,-\alpha])$ and $([\alpha,\omega_1)\times\{-\alpha\})\cup(\{\alpha\}\times(-\omega_1,-\alpha])$.
\end{enumerate}
Further where there are unbounded leaves as described above then $C$ may be chosen so that for any $\alpha\in C$, sets of the form $\{x\}\times[\alpha,\omega_1)$, or appropriate variants of them with coordinates interchanged or multiplied by $-1$, will lie entirely in one leaf of $\mathcal F$ whenever $x\in[-\alpha,\alpha]$.
\end{theorem}
\begin{proof}
It is a matter of fitting together the four quadrants asymptotically foliated according to the two options of Proposition \ref{quadrant constraint}. In order to have a consistent patchwork along the four semi-diagonals, one must keep in mind that the intersection of finitely many closed unbounded sets is again closed and unbounded.
\end{proof}

\begin{figure}[h]
\centering
\scalebox{0.8}{\includegraphics{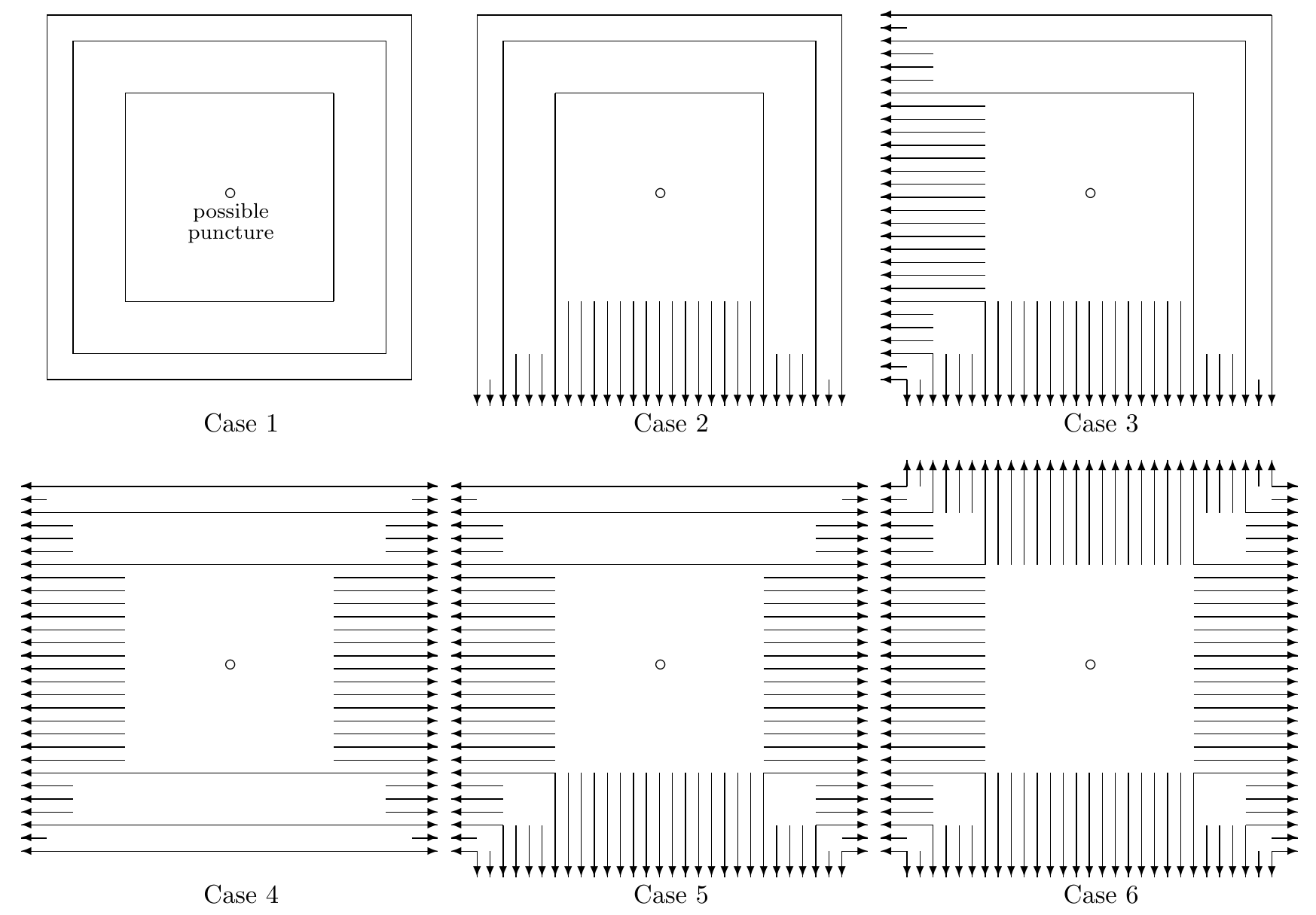}}
\caption{\label{foliateL^2-pt}Foliating ${\mathbb L}^2-\{\mbox{pt}\}$}
\end{figure}

%Old version of the figure that didn't work with arXiv cut to........
\iffalse
\begin{figure}[h]
\centering
    \epsfig{figure=6cases.pdf ,width=130mm}
\caption{\label{foliateL^2-pt}Foliating ${\mathbb L}^2-\{\mbox{pt}\}$}
\end{figure}
\fi
%...................here.

\begin{rem}
{\rm In the first case of Theorem \ref{foliate long plane} we do not have complete freedom to foliate the concentric annuli. As we saw in Remark \ref{foliating cylinders}, there are essentially two ways to foliate an annulus with real lines and two circles so that the boundary components are leaves; see Figure \ref{KneserReebFoliation}. Because the set referred to in Case 1 is closed and unbounded it follows that outside any given square centred on the origin and containing the puncture we can have only finitely many of the annuli containing the Reeb foliation (there may, of course, be infinitely many converging towards the puncture depending on the form of the puncture).}
\end{rem}

We now address the question: how may the regions in which the foliation is not prescribed by Theorem \ref{foliate long plane} be filled. Our first observation is that the regions to be filled fall into seven different categories with or without punctures as follows; refer to Theorem \ref{foliate long plane} and Figure \ref{foliateL^2-pt}.
\begin{enumerate}
\setcounter{enumi}{-1}
\item An annulus each of whose boundary components is a leaf. This arises only in Case 1 (all regions except the central square).
\item A disc whose boundary is a leaf. This arises only as the central square in Case 1.
\item A square in which three sides are foliated by a single leaf while the fourth side is foliated transversely. This arises only as the central square in Case 2.
\item A square in which two adjacent sides are foliated by a single leaf while the other two sides are foliated transversely, but notice that the corner where the latter two sides meet is part of a leaf meeting this square only at that corner. This arises in Case 3 (the central square and bottom left regions), Case 5 (the bottom left hand and right hand regions) and Case 6 (all regions except the central square).
\item A square in which a pair of opposite sides are leaves while the remaining two sides are foliated transversely. This arises in Case 2 (all remaining regions), Case 3 (the top right hand regions), Case 4 (all regions) and Case 5 (the top regions).
\item A square in which one side is a leaf and each of the remaining sides is foliated transversely, again with special attention where two of these sides meet. This arises only in the central square of Case 5.
\item A square in which all four sides are foliated transversely, again with special attention where two of these sides meet. This arises only in the central square of Case 6.
\end{enumerate}

Let us denote these seven partially foliated regions by $C_i$, $i=0,\dots 6$, respectively. 

Possible foliations on $C_0$ and $C_1$ are well-known when there are no punctures: \cite[Prop. 4.2.12, p. 57; Thm 4.2.15, p. 59]{HectorHirschA} provides a complete classification of foliations on the closed annulus (see also \cite{HellmuthKneser24}), while the closed disc cannot be foliated.

\begin{prop}\label{no foliation has a circle}
None of the regions $C_i$ for $2\le i\le 6$ when punctured by removal of a singleton admits a foliated extension of its boundary data in which there occurs a circular leaf.
\end{prop}

\begin{proof} Start by observing that if there is a circular leaf then it must enclose
the puncture in its interior (otherwise we get a foliation on the 2-disc). After removing from $C_i$ the interior of this circular leaf, we obtain a compact surface-with-boundary $W_i$. Let us ``plumb'' together two copies of $W_i$ along the foliated boundary-data (see Figure~\ref{plumb}): the effect of this plumbing is to annihilate the mixed
transverso-tangential behaviour along the boundary to render it purely tangential. Observe that this ``plumbed-double'' denoted $\pi W_i$ is in each case a planar surface-with-boundary, hence a disc with a certain number of holes (say $g_i$). One can even take the double of $\pi W_i$, to obtain a closed surface of genus equal to $g_i$. By counting holes on Figure~\ref{plumb}, we find $g_2=2$, $g_3=3$, $g_4=3$, $g_5=4$ and $g_6=5$. Since this genus is never equal to $1$, the boundary foliation cannot be extended to include a circular leaf.
\end{proof}

\begin{figure}[h]
\centering
\scalebox{0.8}{\includegraphics{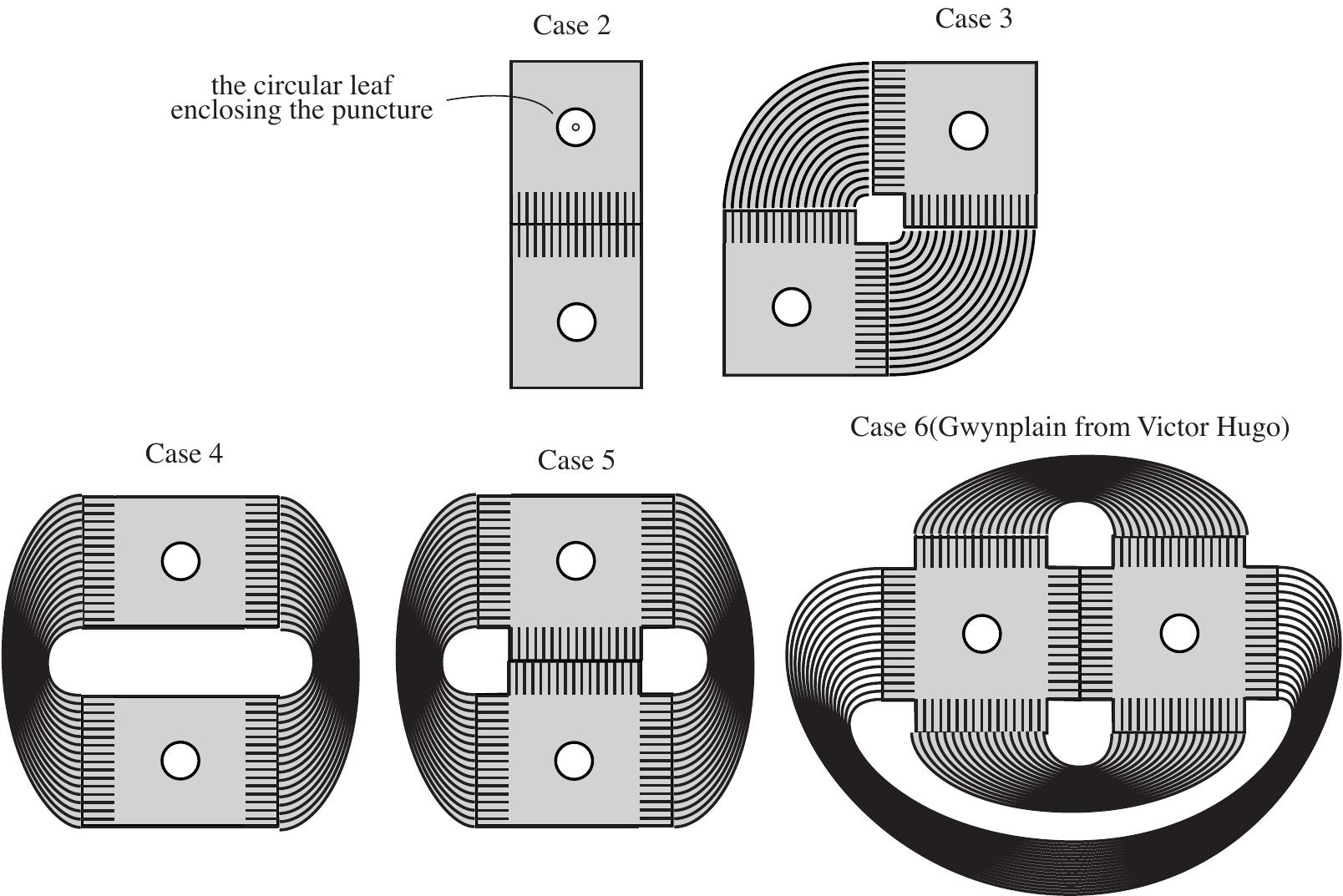}}
    \caption{\label{plumb} Different types of plumbings}
\end{figure}

%Old version of the figure that didn't work with arXiv cut to........
\iffalse
\begin{figure}[h]
\centering
    \epsfig{figure=plumb.pdf ,width=100mm}
    \caption{\label{plumb} Different types of plumbings}
\end{figure}
\fi
%.........................here.

Hence we obtain a complete understanding of the possible leaf types occurring for the sets $C_i$ when punctured by a singleton: \emph{For $i=1$, only $\mathbb S^1$ and ${\mathbb R}$ can occur, while for $2\le i\le6$ only the circle is precluded.} Of
course this is only a coarse overview, still faraway from a complete classification scheme of the topologically distinct foliations.

In classifying foliations on $\mathbb L^2-\{0\}$, one of the main complications arises from the possible occurrences of real leaves (in the form of ``petals'' about the puncture). Petals can be arranged into ``flowers'' with finitely many petals (see Figure~\ref{petal}). In fact one can even observe flowers with countably many petals (of shrinking sizes). Further one can nest many (non-nested) petals inside a given petal, and also plug a Reeb component between two nested petals, etc. Hence the classification scheme looks quite complicated. We do not attempt here to give a complete solution even though this problem resembles Kaplan's classification of foliations on the plane (by means of chordal systems of curves or via non-Hausdorff, second countable, simply connected one-manifolds which is Haefliger-Reeb's point of view) and so can be considered as being ``essentially'' completely solved in the existing literature.

\begin{figure}[h]
\centering
\scalebox{1.1}{\includegraphics{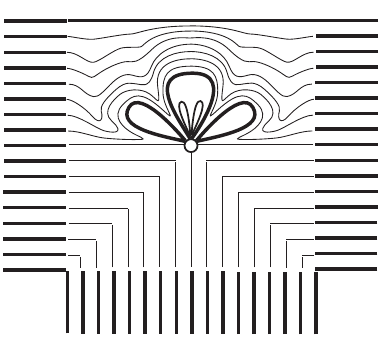}}
    \caption{\label{petal} A menagerie of petals}
\end{figure}

%Old version of the figure that didn't work with arXiv cut to........
\iffalse
\begin{figure}[h]
\centering
    \epsfig{figure=new_menagerie.pdf ,width=45mm}
    \caption{\label{petal} A menagerie of petals}
\end{figure}
\fi
%.............................here.

The plumbing construction also allows us to understand the unpunctured case completely. First a lemma.

\begin{lemma}\label{no turbulence} Let $\mathcal F$ be a foliation on the square $S=[0,1]^2$ extending the following boundary prescriptions: the two horizontal sides $[0,1]\times \{0,1 \}$ of the square are leaves and the foliation is horizontal on thin strips $([0,\varepsilon]\cup [1-\varepsilon, 1]) \times [0,1]$ along the vertical sides (for some immaterial $0<\varepsilon<1/2$). Then there exists a self-homeomorphism $h \colon S \to S$ such that the push-forward foliation $h_{\ast} {\mathcal F}$ becomes the (rectilinear) horizontal foliation.
\end{lemma}

\begin{proof} We divide the proof in three steps.

{\bf Step 1 (Analysis of the possible leaf types via Poincar\'e-Bendixson).} \emph{Each leaf $L$ of $\mathcal F$ is an arc with extremities lying on the opposite sides $\Sigma_0=\{0\} \times I$ and $\Sigma_1=\{1\} \times I$ of the square.}

Since the ambient manifold $S$ is metrisable, a classical chain-argument (of Chevalley-Haefliger, \cite{Haefliger55}) shows that the leaf $L$ endowed with the leaf topology is also second countable. Hence $L$ is homeomorphic to one of the only four possible metrisable one-manifolds (with boundary): namely $S^1$, ${\mathbb R}$, $[0,\infty)$ or $I=[0,1]$. Of course a circular leaf cannot occur (otherwise via Schoenflies we get a foliation on the 2-disc). The two cases ${\mathbb R}$, $[0,\infty)$ are precluded by the Poincar\'e-Bendixson theorem (restrict the foliation to the interior of the square $(0,1)^{2}$ and note that the unbounded-side of $L$, in the sense of having no boundary, cannot escape to infinity.) In conclusion, the only possible leaf-type is $I$. Clearly such a leaf is forced to traverse the square, otherwise it comes back to the same side where it started, and so cuts out a portion of $S$ which when doubled along the boundary via plumbing yields a foliated 2-disc.

{\bf Step 2 (Synchronising a Whitney flow).} To eliminate the transversal behaviour of $\mathcal F$ along the vertical sides of $S$, let us extend the square $S$ to an (infinite) strip $X={\mathbb R}\times I$ over which the foliation is extended horizontally. Since the ambient space $X$ is simply connected, the extended foliation ${\mathcal F}_{\infty}$ on $X$ is orientable, hence describable as the orbits of a flow $\psi\colon {\mathbb R} \times X \to X$ (Whitney \cite{Whitney33}). According to Step 1 (reversing time if necessary), each point $s$ of $\Sigma_0$ will be carried by the flow $\psi$ to a point of $\Sigma_1$ after the elapsing of a certain amount of time $\tau (s)$ (which depends continuously on~$s$). Via the time reparametrisation
$$
\varphi(t,x)=\psi\bigl(\tau(s(x)) t, x\bigr)\,,
$$
where $s(x)$ denotes the unique point of the leaf through $x\in X$ lying on $\Sigma_0$ (whose existence is guaranteed by Step 1), we get a new flow $\varphi$ for which the elapsed time required to traverse from $\Sigma_0$ to $\Sigma_1$ is constantly equal to $1$.

{\bf Step 3 (The synchronised flow $\varphi$ induces a global trivialisation to the horizontal foliation).} By restricting $\varphi\colon {\mathbb R} \times X \to X$ to $[0,1] \times \Sigma_0$, we obtain a map $g\colon [0,1]\times \Sigma_0 \to S$. It is easy to check that $g$ is bijective. For the surjectivity, take $x\in S$. The leaf $L_x$ through $x$ is according to Step 1 a (closed) interval with extremities rooted in different sides. Hence $L_x \cap \Sigma_0$ is a point $s_0$ and by continuity there is a time $t_0\in[0,1]$ such that $\varphi(t_0, s_0)=x$. For the injectivity assume $\varphi(t_1,s_1)=\varphi(t_2,s_2)$, then $s_1=s_2$ (else one obtains a foliation on the disc), and in turn this implies $t_1=t_2$ (otherwise one gets a periodical orbit, circulating again around an impossible foliated disc). Finally when $(t,s)$ moves horizontally in the square $[0,1]\times \Sigma_0$ (i.e. $s$ fixed, $t$ variable), the point $g(t,s)$ describes a specific leaf of $\mathcal F$. Hence the inverse homeomorphism $h=g^{-1}$ takes the foliation ${\mathcal F}$ to the horizontal (straight) foliation on the square.\end{proof}

\begin{cor} \label{longplane_rigididity} Up to homeomorphism there are only two foliations on the long plane ${\mathbb L}^2$ given by the rectilinear models extending Cases 3 and 4 of Theorem \ref{foliate long plane}. Up to isotopy they are only six foliations (four of the ``broken type'' corresponding to Case~3 and two which are the product foliations).
\end{cor}

\begin{proof} We apply Theorem \ref{foliate long plane} and consider the associated sets $C_i$ without punctures. The plumbed doubles of $C_2$, $C_5$ and $C_6$ have non-zero Euler characteristic so by Lemma \ref{Euler} do not allow the corresponding foliations to extend. The case $C_1$ does not lead to a foliation on $\mathbb L^2$ because it would involve a foliated disc. The remaining two cases, $C_3$ and $C_4$, each admit a unique foliated extension according to Lemma~\ref{no turbulence}. The proof of the first clause is now completed by noticing that Lemma \ref{no turbulence} applies as well to all the peripheral (unpunctured) regions arising in Cases~3 and 4 of
Figure \ref{foliateL^2-pt}. The classification up to isotopy follows from the ``super-rigidity'' of the group of self-homeomorphism of ${\mathbb L}^2$ isotopic to the identity map, \cite[Theorem 1.1]{BaillifDeoGauld}.
\end{proof}

\begin{rem}
{\rm The method of proof given in this section may also be extended to other situations. As an example, $\mathbb L^2$ is really just obtained by sewing together eight copies of the first octant $\{(x,y)\in\mathbb L^2 : 0\le y\le x\}$ in a judicious way. We may sew together any finite number of such octants similarly and puncture the outcome to get many more long pipes and related manifolds. As a second example, we may generalise the methods to analyse dimension-one foliations on $\mathbb L^n$ for $n>2$. }
\end{rem}

Let us conclude by giving a simple example vindicating the viewpoint that foliations can be used to distinguish some very similar looking manifolds.

\begin{cor}\label{rectangular v rhombic} The rectangular quadrant ${\bf Q}={\Bbb L}_{\ge
0}\times{\Bbb L}_{\ge 0}$ has a foliation tangent to the
boundary, while the rhombic quadrant ${\overline Q}=\{(x,y) \in {\Bbb
L}^2 : -y \le x\le y \}$ does not. In particular the quadrants
${\bf Q}$ and ${\overline Q}$ are not homeomorphic.
\end{cor}
\begin{proof} Note first that ${\bf Q}$ has a foliation tangent to the
boundary, namely by broken long lines, i.e. sets of the form
$\bigl([\alpha,\omega_1)\times \{\alpha \}\bigr) \cup \bigl(\{
\alpha\} \times [\alpha,\omega_1)\bigr)$ with $\alpha \in
{\Bbb L}_{\ge 0}$.
In contrast ${\overline Q}$ has no such foliation. Indeed according
to Proposition \ref{quadrant constraint} any foliation on ${\overline Q}$ is either
asymptotically horizontal or vertical. In both cases one
observes a ``tripod'' singularity, where a horizontal (resp.
vertical) leaf meets the boundary leaf.
\end{proof}

%\section{Appended technical lemmas}
%\section{Appended technical lemmas}
%\section{Appended technical lemmas}
\section{Appended technical lemmas}

\begin{lemma}\label{Euler}
If a closed topological manifold $M$ carries a dimension-one $C^0$-foliation, then its Euler characteristic $\chi(M)$ vanishes.
\end{lemma}

\begin{proof} Let $\mathcal F$ be a dimension-one foliation on $M$, and assume
$\chi(M)\neq 0$.

Let us first suppose $\mathcal F$ orientable. Then according to Whitney \cite{Whitney33}, there is a flow $f\colon {\mathbb R} \times M \to M$ whose trajectories are exactly the leaves of the foliation. Let $t_n=1/2^n$ denote the dyadic times. We consider the nested sequence $K_n={\rm Fix} (f_{t_n})$ of the fixed-point sets of the dyadic-times of the flow $f_{t_n}$. As the $f_{t_n}$ are all homotopic to the identity map, their Lefschetz numbers $\Lambda(f_{t_n})=\sum_{i} (-1)^i {\rm Trace} H_{i} (f_{t_n})$ all reduces to $\chi(M)$. Since $\chi(M)\neq 0$, the Lefschetz fixed-point theorem tell us that all the $K_n$ are non-void. By compactness of $M$ it follows that $\cap_{n=1}^{\infty} K_n \neq \varnothing$. So there is a point at rest for all dyadic times of the flow, which is then a rest point of the flow. This is a contradiction, as the orbits of Whitney's flow are exactly the leaves of the given foliation.

If the foliation $\mathcal F$ is not orientable, it generates canonically a two-fold cover $M^{\ast} \to M$ making its pull-back ${\mathcal F}^{\ast}$ to $M^{\ast}$ orientable. When the foliation is smooth this cover is just the spherisation of the tangent (line) bundle $T{\mathcal F}$ to the foliation, and in the $C^0$-case this cover is still available thinking in terms of ``germs'', see \cite[p.\,371]{Haefliger62} or \cite[p.\,16--17]{HectorHirschA}. As $\chi(M^{\ast})=2 \chi(M)$, and hence non zero, arguing as above leads to a contradiction. This completes the proof of the lemma.
\end{proof}

\iffalse\begin{rem}
{\rm Whether the converse holds as well is unknown (as far as we
know). Of course when $M$ admits a smooth structure the
converse is a classical result of Hopf \cite{Hopf26}.}
\end{rem}\fi

\begin{lemma}\label{Eulercodim}
If a closed topological manifold $M$ has a codimension-one
$C^0$-foliation, then its Euler characteristic $\chi(M)$
vanishes.
\end{lemma}

\begin{proof} Choose $\mathcal F$ a codimension-one foliation on $M$.
According to Siebenmann \cite[Theorem 6.26, p.\,159]{Siebenmann72} there is a (dimension-one) foliation ${\mathcal F}^{\pitchfork}$ on $M$ transverse to ${\mathcal F}$. The proof is completed by Lemma~\ref{Euler}.
\end{proof}

\begin{rem}
{\rm Note that the converses of Lemmas \ref{Euler} and \ref{Eulercodim} hold when $M$ supports a smooth structure, by results of Hopf \cite{Hopf26} and Thurston \cite{Thurston76}, respectively. Whether these converses extend to topological (closed) manifolds is unclear to the authors as both proofs rely on triangulations. A possible first step towards a positive solution of the Hopf converse might be the construction of a non-singular path field by Brown and Fadell, \cite{BrownFadell}.}
\end{rem}

\begin{prop}\label{prop7.6}
The $n$-ball $\mathbb B^n$ cannot be $C^1$-foliated (neither
tangentially nor transversally and whatever the codimension).
In dimension- or codimension-one the conclusion may be
strengthened with $C^0$ in place of $C^1$.
\end{prop}

\begin{proof}
Assume there is a $p$-dimensional foliation ${\mathcal F}$ on the
ball $\mathbb B^n$. We look simultaneously at the induced
boundary-foliation $\partial {\mathcal F}$ and its double $2{\mathcal
F}$, defined respectively on $\partial \mathbb B^n=\mathbb S^{n-1}$ and
$2\mathbb B^n=\mathbb S^n$.
But recall that even-dimensional spheres $\mathbb S^m$ never admits
$C^1$-foliations. Indeed, according to
Steenrod \cite[Theorem 27.18, p.\,144]{Steenrod}, such
spheres do not even support fields of tangent $p$-planes
($0<p<m$). This gives almost the result, but a careful analysis is nevertheless demanded.

$\bullet$ If $n$ is even, then since $\dim(2 {\mathcal F})=p$, we
have by Steenrod $p=0$ or $p=n$, and we are done.

$\bullet$ If $n$ is odd, we concentrate on the boundary-foliation $\partial {\mathcal F}$ on $\mathbb S^{n-1}$ (an even-dimensional sphere). We distinguish two cases according as $\mathcal F$ behaves tangentially or transversally along the boundary $\partial \mathbb B^n$.

{\it First case (tangential behaviour).} Then $\dim(\partial
{\mathcal F})=p$. Again by Steenrod, we have $p=0$ or $p=n-1$. In
the first case we have finished. In the second,
the boundary $\partial \mathbb B^n$ is a leaf of $\mathcal F$. Hence
$\partial \mathbb B^n=\mathbb S^{n-1}$ is a leaf of $(\mathbb S^n,2 {\mathcal F})$. A
well-known result of Ehresmann-Reeb \cite{Ehresmann-Reeb44},
\cite{Ehresmann51}, tell us that the leaf $\mathbb S^{n-1}$ is
parallelisable. [Indeed, removing a point of $\mathbb S^n$ (not on the
`equator' $\mathbb S^{n-1}$) leave us with $U \approx {\mathbb R}^n$
which is contractible. Hence the tangent plane field to the
induced foliation on $U$ is trivial. Since this bundle
restricts to the tangent bundle of $\mathbb S^{n-1}$, the assertion
follows.] In particular $\chi(\mathbb S^{n-1})=0$\footnote{Actually
much more is true: namely $n=2,4,8$ (Bott-Kervaire-Milnor),
but we don't need this deep information here.}. A
contradiction as $n$ is odd.

{\it Second case (transversal behaviour).} Then $\dim(\partial
{\mathcal F})=p-1$. So still by Steenrod, $p-1$ is $0$ or $n-1$.
In the second case, we have finished. In the first case $p=1$,
and so $2 {\mathcal F}$ is an orientable one-dimensional foliation
(orientability comes from the simple-connectivity of $\mathbb S^n$).
According to Whitney \cite{Whitney33} there is a flow $f\colon
{\mathbb R}\times \mathbb S^n \to \mathbb S^n$ whose orbits describe exactly the
leaves of $2 {\mathcal F}$. Due to the transversal behavior of
${\mathcal F}$ at the boundary, the flow $f$ is pointing into
the same hemisphere along $\partial \mathbb B^n=\mathbb S^{n-1}$.
So by restriction, one obtains a semi-flow
$\varphi\colon{\mathbb R}_{\ge 0} \times \mathbb B^n \to
\mathbb B^n$. By the Brouwer fixed-point theorem (conjointly
with the trick of dyadic cascadization, compare proof of
Lemma~\ref{Euler}), there is a rest point for $\varphi$,
which obviously is also at rest for the flow $f$. But this
contradicts the fact that the trajectories of Whitney's flow
are describing exactly the leaves of the foliation. This
establishes the $C^1$ part of the proposition. 

The $C^0$-strengthening follows from the Siebenmann transversality theorem from \cite{Siebenmann72} and already quoted in the proof of Lemma \ref{Eulercodim}.
\end{proof}

{\it Note.} We do not know whether Proposition~\ref{prop7.6}
holds in general for $C^0$-foliations. This cannot be
straightforwardly reduced to the $C^1$-case, due to the
failure of smoothing $C^0$-foliations (consult \cite{Harrison}).

\begin{minipage}[b]{0.5\linewidth} 
Mathieu Baillif

Universit\'e de Gen\`eve

Section de Math\'ematiques

2-4 rue du Li\`evre, CP 64

CH-1211 Gen\`eve 4

Switzerland.

Mathieu.Baillif@unige.ch
\end{minipage}
\hspace{-25mm} % To get a little bit of space between the figures
\begin{minipage}[b]{0.5\linewidth}
Alexandre Gabard

Universit\'e de Gen\`eve

Section de Math\'ematiques

2-4 rue du Li\`evre, CP 64

CH-1211 Gen\`eve 4

Switzerland.

alexandregabard@hotmail.com
\end{minipage}
\hspace{-25mm}
\begin{minipage}[b]{0.5\linewidth}
David Gauld

Department of Mathematics

The University of Auckland

Private Bag 92019

Auckland

New Zealand.

d.gauld@auckland.ac.nz 
\end{minipage}

\end{document}